\documentclass[12pt]{amsart}
\begin{document}
\numberwithin{equation}{section}

\def\1#1{\overline{#1}}
\def\2#1{\widetilde{#1}}
\def\3#1{\widehat{#1}}
\def\4#1{\mathbb{#1}}
\def\5#1{\frak{#1}}
\def\6#1{{\mathcal{#1}}}

\newcommand{\UH}{\mathbb{H}}
\newcommand{\de}{\partial}
\newcommand{\R}{\mathbb R}
\newcommand{\Ha}{\mathbb H}
\newcommand{\al}{\alpha}
\newcommand{\tr}{\widetilde{\rho}}
\newcommand{\tz}{\widetilde{\zeta}}
\newcommand{\tk}{\widetilde{C}}
\newcommand{\tv}{\widetilde{\varphi}}
\newcommand{\hv}{\hat{\varphi}}
\newcommand{\tu}{\tilde{u}}
\newcommand{\tF}{\tilde{F}}
\newcommand{\debar}{\overline{\de}}
\newcommand{\Z}{\mathbb Z}
\newcommand{\C}{\mathbb C}
\newcommand{\Po}{\mathbb P}
\newcommand{\zbar}{\overline{z}}
\newcommand{\G}{\mathcal{G}}
\newcommand{\So}{\mathcal{S}}
\newcommand{\Ko}{\mathcal{K}}
\newcommand{\U}{\mathcal{U}}
\newcommand{\B}{\mathbb B}
\newcommand{\oB}{\overline{\mathbb B}}
\newcommand{\Cur}{\mathcal D}
\newcommand{\Dis}{\mathcal Dis}
\newcommand{\Levi}{\mathcal L}
\newcommand{\SP}{\mathcal SP}
\newcommand{\Sp}{\mathcal Q}
\newcommand{\A}{\mathcal O^{k+\alpha}(\overline{\mathbb D},\C^n)}
\newcommand{\CA}{\mathcal C^{k+\alpha}(\de{\mathbb D},\C^n)}
\newcommand{\Ma}{\mathcal M}
\newcommand{\Ac}{\mathcal O^{k+\alpha}(\overline{\mathbb D},\C^{n}\times\C^{n-1})}
\newcommand{\Acc}{\mathcal O^{k-1+\alpha}(\overline{\mathbb D},\C)}
\newcommand{\Acr}{\mathcal O^{k+\alpha}(\overline{\mathbb D},\R^{n})}
\newcommand{\Co}{\mathcal C}
\newcommand{\Hol}{{\sf Hol}(\mathbb H, \mathbb C)}
\newcommand{\Aut}{{\sf Aut}(\mathbb D)}
\newcommand{\D}{\mathbb D}
\newcommand{\oD}{\overline{\mathbb D}}
\newcommand{\oX}{\overline{X}}
\newcommand{\loc}{L^1_{\rm{loc}}}
\newcommand{\la}{\langle}
\newcommand{\ra}{\rangle}
\newcommand{\thh}{\tilde{h}}
\newcommand{\N}{\mathbb N}
\newcommand{\kd}{\kappa_D}
\newcommand{\Hr}{\mathbb H}
\newcommand{\ps}{{\sf Psh}}
\newcommand{\Hess}{{\sf Hess}}
\newcommand{\subh}{{\sf subh}}
\newcommand{\harm}{{\sf harm}}
\newcommand{\ph}{{\sf Ph}}
\newcommand{\tl}{\tilde{\lambda}}
\newcommand{\gdot}{\stackrel{\cdot}{g}}
\newcommand{\gddot}{\stackrel{\cdot\cdot}{g}}
\newcommand{\fdot}{\stackrel{\cdot}{f}}
\newcommand{\fddot}{\stackrel{\cdot\cdot}{f}}

\def\Re{{\sf Re}\,}
\def\Im{{\sf Im}\,}

\newcommand{\Real}{\mathbb{R}}
\newcommand{\Natural}{\mathbb{N}}
\newcommand{\Complex}{\mathbb{C}}
\newcommand{\ComplexE}{\overline{\mathbb{C}}}
\newcommand{\Int}{\mathbb{Z}}
\newcommand{\UD}{\mathbb{D}}
\newcommand{\clS}{\mathcal{S}}
\newcommand{\gtz}{\ge0}
\newcommand{\gt}{\ge}
\newcommand{\lt}{\le}
\newcommand{\fami}[1]{(#1_{s,t})}
\newcommand{\famc}[1]{(#1_t)}
\newcommand{\ts}{t\gt s\gtz}
\newcommand{\classCC}{\tilde{\mathcal C}}
\newcommand{\classS}{\mathcal S}

\newcommand{\Step}[2]{\begin{itemize}\item[{\bf Step~#1.}]{\it #2}\end{itemize}}
\newcommand{\step}[2]{\begin{itemize}\item[{\it Step~#1.}]{\it #2}\end{itemize}}
\newcommand{\proofbox}{\hfill$\Box$}

\newcommand{\mcite}[1]{\csname b@#1\endcsname}
\newcommand{\UC}{\mathbb{T}}

\newcommand{\Moeb}{\mathrm{M\ddot ob}}

\newcommand{\dAlg}{{\mathcal A}(\UD)}
\newcommand{\diam}{\mathrm{diam}}

\theoremstyle{theorem}
\newtheorem {result} {Theorem}
\setcounter {result} {64}
 \renewcommand{\theresult}{\char\arabic{result}}



\newcommand{\Spec}{\Lambda^d}
\newcommand{\SpecR}{\Lambda^d_R}
\newcommand{\Prend}{\mathrm P}




\def\cn{{\C^n}}
\def\cnn{{\C^{n'}}}
\def\ocn{\2{\C^n}}
\def\ocnn{\2{\C^{n'}}}
\def\je{{\6J}}
\def\jep{{\6J}_{p,p'}}
\def\th{\tilde{h}}


\def\dist{{\rm dist}}
\def\const{{\rm const}}
\def\rk{{\rm rank\,}}
\def\id{{\sf id}}
\def\aut{{\sf aut}}
\def\Aut{{\sf Aut}}
\def\CR{{\rm CR}}
\def\GL{{\sf GL}}
\def\Re{{\sf Re}\,}
\def\Im{{\sf Im}\,}
\def\U{{\sf U}}

\def\la{\langle}
\def\ra{\rangle}

\emergencystretch15pt \frenchspacing

\newtheorem{theorem}{Theorem}[section]
\newtheorem{lemma}[theorem]{Lemma}
\newtheorem{proposition}[theorem]{Proposition}
\newtheorem{corollary}[theorem]{Corollary}

\theoremstyle{definition}
\newtheorem{definition}[theorem]{Definition}
\newtheorem{example}[theorem]{Example}

\theoremstyle{remark}
\newtheorem{remark}[theorem]{Remark}
\numberwithin{equation}{section}

\title[Chordal Loewner chains]{Geometry behind chordal Loewner chains}

\author[M. D. Contreras]{Manuel D. Contreras $^\dag$}

\author[S. D\'{\i}az-Madrigal]{Santiago D\'{\i}az-Madrigal $^\dag$}
\address{Camino de los Descubrimientos, s/n\\
Departamento de Matem\'{a}tica Aplicada II\\
Escuela T\'{e}cnica Superior de Ingenieros\\
Universidad de Sevilla\\
Sevilla, 41092\\
Spain.}\email{contreras@us.es} \email{madrigal@us.es}

\author[P. Gumenyuk]{Pavel Gumenyuk $^\ddag$}
\address{Department of
Mathematics\\ University of Bergen, Johannes Brunsgate 12\\ Bergen 5008, Norway. }
\email{Pavel.Gumenyuk@math.uib.no}

\date{\today }


\thanks{$^\dag$ Partially supported by the \textit{Ministerio
de Ciencia e Innovaci\'on} and the European Union (FEDER), projects MTM2006-14449-C02-01 and
MTM2009-14694-C02-02, by \textit{La Consejer\'{\i}a de Educaci\'{o}n y Ciencia de la Junta de Andaluc\'{\i}a}
and by the ESF Networking Programme ``Harmonic and Complex Analysis and its Applications''}

\thanks{$^\ddag$ Partially supported by the ESF Networking Programme
``Harmonic and Complex Analysis and its Applications'' and the {\it Research Council of Norway}}

\begin{abstract}

Loewner Theory is a deep technique in Complex Analysis affording a basis for many further
important developments such as the proof of famous Bieberbach's conjecture and well-celebrated
Schramm's Stochastic Loewner Evolution~(SLE). It provides analytic description of expanding
domains dynamics in the plane. Two cases have been developed in the classical theory, namely
the {\it radial} and the {\it chordal} Loewner evolutions, referring to the associated
families of holomorphic self-mappings being normalized at an internal or boundary point of the
reference domain, respectively. Recently there has been introduced a new approach~\cite{BCM1,
BCM2, SMP} bringing together, and containing as quite special cases, radial and chordal
variants of Loewner Theory. In the framework of this approach we address the question what
kind of systems of simply connected domains can be described by means of Loewner chains of
chordal type. As an answer to this question we establish a necessary and sufficient condition
for a set of simply connected domains to be the range of a generalized Loewner chain of
chordal type. We also provide an easy-to-check geometric sufficient condition for that. In
addition, we obtain analogous results for the less general case of chordal Loewner evolution
considered  in~\cite{Aleks1983, Goryainov-Ba, Bauer}.

\end{abstract}

\maketitle

\tableofcontents

\section{Introduction}
\subsection{Parametric representation}
Loewner theory is a deep technique in Geometric Function Theory. In 1923 C.\,Loewner\footnote{{\it Charles
Loewner} was a Czech--American mathematician, also known as {\it Karel L\"owner} (Czech) and {\it Karl L\"owner}
(German).} published a paper~\cite{Loewner} where he introduced the so-called {\it parametric representation} of
slit mappings, i.e. univalent holomorphic functions mapping the unit disk~$\UD:=\{z:|z|<1\}$ onto the complex
plane minus a slit along a Jordan curve extending to infinity. His original aim was to achieve some progress in
the famous Bieberbach conjecture on the sharp upper bounds for the Taylor coefficients of normalized univalent
functions. The information about a slit mapping~$f$ is encoded in a continuous function $u:[0,+\infty)\to\Real$
in such a way that the function~$f$ can be expressed via integrals of the following {\it Loewner ODE}
\begin{equation}\label{L_ODE_intro}
\frac{dw}{dt}=-w\frac{e^{iu(t)}+w}{e^{iu(t)}-w},\quad
t\ge0;~~~w|_{t=0}=z,~z\in\UD.
\end{equation}

The modern form of the Parametric Representation Method, which provides a tool to represent the whole
class~$\classS$ consisting of all univalent holomorphic functions $f:\UD\to\Complex$ normalized by the
expansion $f(z)=z+a_2z^2+\ldots$, is due to fundamental contributions of Kufarev~\cite{Kufarev} and
Pommerenke~\cite{Pommerenke-65}. The Schwarz kernel in the right-hand side of~\eqref{L_ODE_intro} is
replaced then by an arbitrary holomorphic function $p(w,t)$ measurable in~$t$ and satisfying conditions
$\Re p(w,t)>0$ and $p(0,t)=1$ for all $w\in\UD$, $t\ge0$. More details can be found
in~\cite[Chapter~6]{Pommerenke}. An exposition for the case of slit mappings, essentially the same as
originally considered by Loewner, is given in~\cite[\S III.2]{Goluzin}. Other good source
is~\cite[Chapter~3]{Duren} and the monograph~\cite{Aleksandrov}, which is devoted to the Parametric
Representation Method and its applications to extremal problems of Geometric Function Theory.

It is worth mentioning that the Parametric Representation Method, providing one of the main ingredients in the
proof of Bieberbach's conjecture given by de Branges~\cite{deBranges} in 1984, has however gone far beyond the
scope of the initial problem. Two spectacular examples of its protrusion to other areas of mathematics are the
Loewner--Kufarev type equation describing free boundary flow of a viscous fluid in a Hele-Shaw
cell~\cite{VinKuf}, see also \cite[\S1.4.2]{GusVas}; and the highly celebrated Stochastic Loewner Evolution
(SLE) introduced in 2000 by Schramm~\cite{Schramm} as a powerful tool that led to deep results in the
mathematical theory of some 2D lattice models of great importance for Statistical Physics. Recently Markina,
Prokhorov, and Vasil'ev~\cite{MPV1,MPV2} have discovered interesting relations between Loewner Theory,
Integrable Systems, and Kirillov's representation of the Virasoro algebra, which constitutes a common algebraic
skeleton for such topics in Mathematical Physics as the KdV non-linear PDE and String Theory.

Originally Loewner Theory was developed for univalent functions in the unit disk normalized at the origin.
Parametric representation of univalent functions in the upper half-plane having hydrodynamical
normalization at the point of infinity was developed by Kufarev and his students, see
e.\,g.,~\cite{Kufarev_etal}. The analogue of Loewner ODE~\eqref{L_ODE_intro} for this case appeared
probably for the first time in papers~\cite{Kufarev1946, Popova1949}. Further development in this
direction was made in~\cite{Aleks1983, AleksST, Goryainov-Ba, Bauer}.

In modern literature the case of univalent functions normalized at an internal point of the reference domain is
referred to as the {\it radial} case. The parametric representation of univalent functions with hydrodynamical
normalization, in contrast, is marked by the attribute {\it ``chordal"}. The latter  of these two cases had been
quite underestimated until the already mentioned paper~\cite{Schramm} by Schramm, where he used both chordal and
radial variants of Loewner Theory to introduce the notion of SLE.

Another mathematical construction closely related to Loewner Theory is one-parametric semigroups of holomorphic
functions. We will not discuss it in details in this paper, so we just mention that this construction provides a
set of important non-trivial examples for Operator Theory (see, e.g., \cite{Siskakis}) and Markov stochastic
processes and refer the reader to the monographs~\cite{Abate, Shapiro, Shoikhet} for the definition and further
details, and to papers~\cite{Goryainov1, Goryainov2} where some applications to Probability Theory were
developed. An exposition of this topic from the dynamical point of view can be found
in~\cite{Contreras-Diaz:pacific}.

Recently Bracci, Contreras, and D\'\i{}az-Madrigal~\cite{BCM1,BCM2} have  introduced a new approach in Loewner
Theory in order to develop a general construction, which contains, as quite special cases, the radial and
chordal Loewner evolutions as well as one-parametric semigroups of holomorphic functions. Now we give a short
account on this approach.

\subsection{New approach in Loewner Theory}

According to the new approach in the Loewner Theory introduced by Bracci, Contreras, and
D\'{\i}az-Madrigal~\cite{BCM1,BCM2}, the essence of this theory can be represented by the relations and
interplay between three notions: {\it Loewner chains, evolution families} and {\it Herglotz vector fields.}

In~\cite{BCM1} the following definition of a (generalized) evolution family was given.
\begin{definition}\label{def-ev}
A family $(\varphi_{s,t})_{0\leq s\leq t<+\infty}$ of holomorphic self-maps of the unit disc  is an {\sl
evolution family of order $d$} with $d\in [1,+\infty]$ (in short, an {\it $L^d$-evolution family}) if
\begin{enumerate}
\item[EF1.] $\varphi_{s,s}=id_{\mathbb{D}},$

\item[EF2.] $\varphi_{s,t}=\varphi_{u,t}\circ\varphi_{s,u}$ whenever $0\leq
s\leq u\leq t<+\infty,$

\item[EF3.] for all $z\in\mathbb{D}$ and for all $T>0$ there exists a
non-negative ${k_{z,T}\in L^{d}([0,T],\mathbb{R})}$ such that
\[
|\varphi_{s,u}(z)-\varphi_{s,t}(z)|\leq\int_{u}^{t}k_{z,T}(\xi)d\xi
\]
whenever $0\leq s\leq u\leq t\leq T.$
\end{enumerate}
\end{definition}
We point out that the elements of any evolution family are univalent~\cite[Corollary 6.3]{BCM1}. The evolution
families corresponding to the classical radial case are those satisfying the normalization
$\varphi_{0,t}=e^{-t}(z+a_2(t)z^2+\ldots\,)$, $t\ge0$, $z\in\UD$. Assuming this normalization one can omit
condition~EF3 in the definition \cite[Lemma 6.1]{Pommerenke}. However, this is obviously not the case in
general. Particular notions appeared earlier in the literature, see, e.g., \cite{Goryainov, Goryainov1996,
Goryainov-Ba, Bauer, Conway2}.

Generalized evolution families can be seen as solutions to the following initial value problem
\begin{equation*}
    \left\{
    \begin{array}{ll}
      \displaystyle\frac{dw}{dt}=G(w,t), & \quad t\in[s,+\infty), \\[5mm]
     \displaystyle w|_{t=s}=z, &
    \end{array}
    \right.
\end{equation*}
where $G(z,t)$ is a Herglotz vector field. For the definition of Herglotz vector fields and further details we
refer the reader to~\cite{BCM1}. Here we only mention the fact~\cite[Theorem 4.8]{BCM1} that Herglotz vector
fields can be characterized by the representation~$$G(z,t)=(\tau(t)-z)\big(1-\overline{\tau (t)}z\big)p(z,t),$$
where $\tau:[0,+\infty)\to\overline{\UD}$ is a measurable function and $p:\UD\times[0,\infty)\to\Complex$ is a
function with non-negative real part, holomorphic in~$z$ and locally integrable in~$t$. The classical radial
case corresponds to $\tau\equiv0$ and $p(z,t)$ satisfying~$p(0,t)=1$, $t\ge0$. If we consider the chordal case
in the unit disk, then~$\tau=1$. More generally, it follows from~\cite[Thereom~6.7]{BCM1} and the proof
of~\cite[Thereom~7.1]{BCM1} that all non-identical elements of an evolution family have the same Denjoy--Wolff
point~$\tau_0$ if and only if for the corresponding Herglotz vector field we have~$\tau\equiv\tau_0$. This is
the reason for the following definition.

\begin{definition}
Let~$(\varphi_{s,t})$ be an evolution family. Suppose that all non-identical elements of $(\varphi_{s,t})$ share
the same Denjoy--Wolff point~$\tau_0\in\overline\UD$, i.e., $\varphi_{s,t}(\tau_0)=\tau_0$ and
$|\varphi_{s,t}'(\tau_0)|\le1$ for all $s\ge0$ and $t\ge s$, where $\varphi_{s,t}(\tau_0)$ and
$\varphi'_{s,t}(\tau_0)$ are to be understood as the corresponding angular limits if
$\tau_0\in\UC:=\partial\UD$. We will say that $(\varphi_{s,t})$ is a {\it radial\,} evolution family if
$\tau_0\in\UD$. Otherwise, if $\tau_0\in\UC$, we will call $(\varphi_{s,t})$ a {\it chordal\,} evolution family.
\end{definition}

By means of M\"obius transformations we can always assume that $\tau_0=0$ in the case of a radial
evolution family and that $\tau_0=1$ in the case of a chordal evolution family.

The general version of the third fundamental notion in Loewner Theory, the notion of generalized Loewner
chains was given in~\cite{SMP}.

\begin{definition}\label{def-lc}
A family $(f_t)_{0\leq t<+\infty}$ of holomorphic maps of the unit disc~$\UD$ is called a {\it Loewner
chain of order $d$} with $d\in [1,+\infty]$ (in short, an {\it $L^d$-Loewner chain}) if

\begin{enumerate}
\item[LC1.] each function $f_t:\D\to\C$ is univalent,

\item[LC2.] $f_s(\D)\subset f_t(\D)$ for all $s\ge0$ and $t\ge s$,

\item[LC3.] for any compact set $K\subset\mathbb{D}$ and any $T>0$ there exists a
non-negative function $k_{K,T}\in L^{d}([0,T],\mathbb{R})$ such that
\[
|f_s(z)-f_t(z)|\leq\int_{s}^{t}k_{K,T}(\xi)d\xi
\]
for all $z\in K$ and all $0\leq s\leq t\leq T$.
\end{enumerate}
A Loewner chain $\famc f$ is said to be {\it normalized} if  $f_0(0)=0$ and $f'_0(0)=1$ (notice that we
only normalize the function $f_0$).
\end{definition}
We note that the classical notion of (radial) Loewner chain is recovered if one replaces condition~LC3 by the
requirement that $f_t$ should satisfy the following normalization: $f_t(z)=e^t(z+c_2(t)z^2+\ldots\,)$.

In~\cite{SMP} we established a deep relation between evolution families and Loewner chains similar to that
taking place in the classical case. These results can be stated in the following three theorems.

\begin{result}[{\cite[Theorem 1.3]{SMP}}]\label{T1}
For any Loewner chain $(f_t)$ of order $d\in[1,+\infty]$, if we define
\begin{equation}\label{EFofLC}\varphi_{s,t}:= f_t^{-1}\circ f_s, \quad 0\le s\le t, \end{equation}
then $(\varphi_{s,t})$ is an evolution family of the same order~$d$. Conversely, for any evolution family
$(\varphi_{s,t})$ of order~$d\in[1,+\infty]$, there exists a Loewner chain $(f_t)$ of the same order~$d$
such that the following equality holds
\begin{equation}\label{main_EV_LC}
 f_t\circ\varphi_{s,t}=f_s,\quad 0\le s\le t.
\end{equation}
\end{result}

\begin{definition}
A Loewner chain $\famc f$ is said to be {\it associated with} an evolution family~$\fami\varphi$ if it
satisfies~\eqref{main_EV_LC}. Another way to express the same fact is to say that $(\varphi_{s,t})$ is the
{\it evolution family of} the Loewner chain~$(f_t)$.
\end{definition}

In general, given an evolution family $(\varphi_{s,t})$, algebraic condition~\eqref{main_EV_LC} does not define
a unique Loewner chain. In fact, there can be a plenty of different Loewner chains associated with the same
evolution family. The following theorem gives necessary and sufficient conditions for the uniqueness of a
normalized Loewner chain associated with a given evolution family.

\begin{result}[{\cite[Theorem 1.6]{SMP}}]\label{T2}
Let $(\varphi_{s,t})$ be an evolution family. Then there exists a unique normalized Loewner chain $(f_t)$
associated with~$(\varphi_{s,t})$ such that $\cup_{t\ge0}f_t(\UD)$ is either an Euclidean disk centered at the
origin or the whole complex plane~$\C$. Moreover, the following statements are equivalent:
\begin{itemize}
\item[(i)] the family $(f_t)$ is the only normalized Loewner chain
associated with the evolution family~$(\varphi_{s,t})$;
\item[(ii)] there exist $z\in\UD$ such that
$$\lim_{t\to+\infty}\frac{|\varphi_{0,t}'(z)|}{1-|\varphi_{0,t}(z)|^2}=0;$$
\item[(iii)] the above limit vanishes for all $z\in\UD$;
\item[(iv)] $\bigcup\limits_{t\gtz}f_t(\UD)=\C$.
\end{itemize}
\end{result}
We call the family $(f_t)$ in the above theorem the {\it standard Loewner chain} associated with the
evolution family~$(\varphi_{s,t})$.

The following theorem explains what happens  when an associated Loewner chain is not unique. As above we
will  denote by $\clS$ the class of all univalent holomorphic functions $h$ in the unit disk~$\UD$,
normalized by $h(0)=h'(0)-1=0$.

\begin{result}[{\cite[Theorem 1.7]{SMP}}]\label{T3}
Let $(\varphi_{s,t})$ be an evolution family and $(f_t)$ the standard Loewner chain associated
with~$(\varphi_{s,t})$. Suppose that $$\Omega:=\bigcup\limits_{t\gtz}f_t(\UD)\neq\C$$ and write
$\beta:=\lim_{t\to+\infty}\frac{|\varphi_{0,t}'(0)|}{1-|\varphi_{0,t}(0)|^2}$ (such a limit always exists). Then
$\Omega=\{z:|z|<1/\beta\}$ and for any other normalized Loewner chain $(g_t)$ associated with the evolution
family~$(\varphi_{s,t})$, there is $h\in\clS$ such that
$$
g_t(z)=h\big(\beta f_t(z)\big)/\beta. $$
\end{result}

Using the relation between Loewner chains and evolution families we can now introduce the following
definition.

\begin{definition}\label{RadChLoewnerChain}
Let $(f_t)$ be a Loewner chain. Suppose that the functions $\varphi_{s,t}:=f_t^{-1}\circ f_s$, $t\ge
s\ge0$, form a radial (respectively, chordal) evolution family. In this case we will say that the Loewner
chain $(f_t)$ is of {\it radial} (respectively, {\it chordal}) {\it type}.
\end{definition}

\subsection{Problem, definition, and statement of the main results} The general problem this paper is devoted to study the
relation between geometric properties of Loewner chains and analytic properties of their evolution
families. More specifically, our aim is to give a complete characterization for the systems of image
domains $f_t(\UD)$ generated by Loewner chains whose evolution family belongs to one or another certain
class, including that of all radial and chordal evolution families.

As we will see a little bit later, for the case of radial evolution families this problem turns out to be solved
by means of simple modifications of some proofs in the classical Loewner Theory. At the same time the chordal
case seems to be much harder. The same can be said concerning the class of evolution families introduced
in~\cite{Goryainov}, which we consider in Section~\ref{Goryainov-Ba}.

To formulate our main results we introduce the following definition.
\begin{definition}
\label{inclusion_chain} Let $(\Omega_t)_{t\ge0}$ be a family of simply connected domains
$\Omega_t\subset\C$ and $\Omega_t \neq \C$ for all $t\in[0,+\infty)$. The family $(\Omega_t)_{t\ge0}$ is
said to be a \textit{inclusion chain} if it satisfies the following conditions:

\begin{itemize}
\item[IC1.] $\Omega_s\subset\Omega_t$ whenever $t\ge s\ge0$.

\item[IC2.] $\Omega_{t_n}\to\Omega_{t_0}$ whenever $t_n\to t_0$
in the sense of the kernel convergence  w.r.t. any point of~$\Omega_0$.
\end{itemize}
\end{definition}
For the notion of kernel convergence of domains we refer the reader to \cite[\S II.5]{Goluzin} or
\cite[Section 1.4]{Pommerenke}.
\begin{remark}\label{inclusion_chain_Remark}
We note that, as it follows from the definition of the kernel of a sequence of domains,  condition IC2
above can replaced by the two following simple conditions of topological nature:
\begin{itemize}
\item[(i)] for each $t\ge0$ the domain $\Omega_t$ is a connected component of the interior
of~$\bigcap_{s>t}\Omega_s$,
\item[(ii)] for each $t>0$ the domain $\Omega_t$ coincides with $\bigcup_{s\in[0,t)}\Omega_s$.
\end{itemize}
\end{remark}

Denote by $r(D,z)$ the conformal radius of a simply connected domain~$D$ w.r.t. a point~$z$. The following
two theorems explain our motivation to introduce the notion of inclusion chain.  As far as we know, these
theorems have not been stated previously but we think they can be considered as well-known among the
specialists. In any case, for the sake of completeness we will sketch their proofs in
Section~\ref{InclusionsversusLoewner}.

\begin{theorem}
\label{characterizarioninclusionchain1} Let $(\Omega_t)_{t\ge0}$ be a family of simply connected domains
$\Omega_t\subset\C$. Suppose that $\Omega_s\subset\Omega_t \neq \C$ for all  $s\ge0$ and $t\ge s$. Then
the following assertions are equivalent:
\begin{enumerate}
\item The family $(\Omega_t)_{t\ge0}$ is an inclusion chain.
\item There exists $w_{\ast }\in \Omega _{0}$ such that the
function $t\in [0,+\infty) \mapsto \mu _{w_{\ast }}(t):=r(\Omega _{t},w_{\ast })$ is continuous.
\item For all $w_{\ast }\in \Omega _{0}$ the
function $t\in [0,+\infty) \mapsto \mu _{w_{\ast }}(t):=r(\Omega _{t},w_{\ast })$ is non-decreasing and
continuous.
\end{enumerate}
\end{theorem}

\begin{theorem}
\label{characterizarioninclusionchain2} Under conditions of Theorem~\ref{characterizarioninclusionchain1}
the following two assertions are equivalent:
\begin{enumerate}
\item The family $(\Omega_t)_{t\ge0}$ is a inclusion chain.
\item There exists a Loewner chain $(f_t)$  such that
\begin{equation*}
\label{globalEq}\{\Omega_t:t\geq 0\}=\{f_t(\D):t\geq0\}.
\end{equation*}
\end{enumerate}
\end{theorem}

In the proofs of the above two theorems, the Loewner chain that we construct for a given inclusion chain
$(\Omega _{t})_{t\geq 0}$, is of radial type. This is not surprising, because in \cite{SMP} (proof of
Theorem~3.3) we showed that given an arbitrary Loewner chain $(f_t)$, then there is a Loewner chain
$(g_t)$ of radial type such that $f_{t}(\mathbb{D})=g_{t}(\mathbb{D})$ for all $t\ge0$. However, this
statement does not hold any more if one requires $(g_t)$ to be a Loewner chain of chordal type instead of
that of radial type. This can be easily seen if we consider any inclusion chain $(\Omega_t)$ having the
following property: there exists $s,t\ge0$ such that $\partial\Omega_s\cap\partial\Omega_t=\emptyset$.

So we address the following problem.\nopagebreak\nopagebreak
\begin{trivlist}\vskip2mm
\item\noindent{\bf Problem.} {\it Characterize all inclusion chains~$(\Omega_t)$ which are the ranges of Loewner chains of
chordal type, i.e., $\{\Omega_t:t\ge0\}=\{f_t(\UD):t\ge0\}$ for some Loewner chain~$(f_t)$ such that the functions $\varphi_{s,t}=f_t^{-1}\circ f_s$, $t\ge s\ge0$, form a chordal evolution family.}
\vskip2mm
\end{trivlist}

One can regard as known the fact that {\it the class of all inclusion chains satisfying the condition in the
above problem contains all ``slit'' inclusion chains, i.e. inclusion chains obtained by gradual removing a slit
(or a finite system of slits) in a simply connected domain,} see Section~\ref{the example}. Our main result
concerns the general case and can be formulated as follows. Denote by $\Prend(D)$ the Carath\'eodory boundary of
a domain $D$, i.e. the set of all its prime ends.

\begin{theorem}
\label{charactezationchordalinclusionchain} Let $(\Omega _{t})_{t\geq 0}$ be an inclusion chain. Then the
following statements are equivalent:
\begin{enumerate}
     \item There exists a prime end~$P_0\in\Prend(\Omega_0)$ such that for each $t\ge0$ the
domain $\Omega_0$ is embedded into the domain~$\Omega_t$ conformally at the prime end~$P_0$.
     \item There is a Loewner chain $(f_t)$ of chordal type  such that
$$\{\Omega_t:t\geq 0\}=\{f_t(\D):t\geq0\}.$$
\end{enumerate}
\end{theorem}
For the notion of domain embedded conformally at a prime end, see Definitions~\ref{inD} and~\ref{Conf_emb} in
Section~\ref{main-section}.

The  above  theorem is a direct consequence of the more technical Theorem~\ref{Main_Theorem}. It is worth
mentioning that the Loewner chain~$(f_t)$ in this theorem can be always assumed to be of order $d=\infty$.

In Section~\ref{Goryainov-Ba} we state and proof a result (Theorem~\ref{GorBaTheorem}), which is an
analogue of Theorem~\ref{charactezationchordalinclusionchain} for a subclass of chordal evolution families
having some additional regularity, which were considered in~\cite{Goryainov-Ba}.

Theorems~\ref{charactezationchordalinclusionchain} and~\ref{GorBaTheorem} give necessary and sufficient
conditions for an inclusion chain to be the range of a Loewner chain of given type. However, in general these
conditions might be quite difficult to verify.  Therefore  in Section~\ref{Geometric properties} we give some
easy-to-check geometric sufficient conditions.

Finally in Sections~\ref{the example} and~\ref{EFinDA} we consider the question {\it how the boundary behaviour
of an evolution family affects the geometry of the associated Loewner chains}. It turns out
(Proposition~\ref{ex_prop}) that there exists an evolution family consisting of functions with continuous
extensions to the unit circle such that none of the associated Loewner chains has image domains with locally
connected boundaries. At the same time, as we prove in Section~\ref{EFinDA}, if a Loewner chain $(f_t)$ has the
property that $\partial f_t(\UD)$ is locally connected for all~$t\ge0$, then all the functions in the evolution
family of this Loewner chain can be continuously extended to the unit circle.

It is also worth mentioning that further in the paper we prefer to work with the notion of L-admissible family
instead of that of inclusion chain, see Definition~\ref{def_L-adm} in the next section. In some cases, however,
it obviously makes absolutely no difference  and many results hold with words ``L-admissible family`` replaced
by ``inclusion chain''.

\section{Inclusions chains versus Loewner chains}
\label{InclusionsversusLoewner}
In this section we will sketch the proofs of Theorems~\ref{characterizarioninclusionchain1} and~\ref{characterizarioninclusionchain2}. Let us start with some auxilliary  statements.

\begin{lemma} Let $\varphi$ be a holomorphic univalent self-mapping of the unit disk~$\UD$ with ${\varphi(0)=0}$. Then:
\begin{enumerate}
  \item for all $z\in\D$,
  $$|\varphi(z)-z|\leq |1-\varphi'(0)|+\sqrt{1-|\varphi'(0)|^2};$$
  \item if $\varphi'(0)>0$, then for all $z\in\D$,
  $$|\varphi(z)-z|\leq 3\sqrt{1-\varphi'(0)}.$$
\end{enumerate}
\end{lemma}
\begin{proof}
Write $\varphi (z)=\sum _{n=1}^{\infty} a_n z^n$. Let us recall that $$ \sum _{n=1}^{\infty}
n|a_n|^2=m(\varphi(\D))\leq m(\D)=1,$$ where $m(\cdot)$ stands for the normalized Lebesgue measure in the unit
disc. Let ${\phi(z):=\varphi(z)-a_1 z}$. Therefore
\begin{eqnarray*}
|\varphi(z)-z| &\leq & \left| \sum _{n=2}^{\infty} a_nz^n\right|+|1-\varphi'(0)| \leq \|\phi\|_{H_\infty}
+|1-\varphi'(0)| \leq \|\phi\|_{H_2} +|1-\varphi'(0)| \\
   &=&  \sqrt{\sum _{n=2}^{\infty} |a_n|^2} +|1-\varphi'(0)|\leq \sqrt{\sum
_{n=2}^{\infty} n|a_n|^2} +|1-\varphi'(0)|\\
    &\leq&
\sqrt{1-|\varphi'(0)|^2} +|1-\varphi'(0)|.
\end{eqnarray*}
This proves (1). Assertion (2) is now an immediate consequence of (1).
\end{proof}

\begin{lemma}
\label{subordination} Let $f,g:\D\to\C$ be two univalent functions such that
\begin{enumerate}
  \item $f(0)=g(0)$,
  \item $f(\D)\subset g(\D)$,
  \item $f'(0),g'(0)>0$.
\end{enumerate}
Then for all $r\in(0,1)$ there exists a constant $C=C(r)>0$ not depending on the functions~$f$ and~$g$
such that $$|f(z)-g(z)|\leq C\sqrt{g'(0)(g'(0)-f'(0))}, \quad \textrm{ for all } |z|\leq r.$$
\end{lemma}
\begin{proof}
Write $\varphi=g^{-1}\circ f$ and fix $r<1$. Notice that $\varphi (D(0,r))\subseteq D(0,r)$. Take any $z$
with $|z|\leq r$ and denote by $I$ the line segment joining the points $z$ and $\varphi(z)$. Then using
the Koebe distortion theorem~(see, e.g., \cite[\S II.4]{Goluzin} or \cite[Theorem~1.6 on
p.\,21]{Pommerenke}) we get
\begin{eqnarray*}
  |f(z)-g(z)| &=& |g(\varphi(z))-g(z)|=\left|\int_z^{\varphi(z)}g'(\xi)d\xi\right|
   \leq |\varphi(z)-z|\,\max_{\xi\in I}|g'(\xi)|\\ &\leq&  |\varphi(z)-z|\, g'(0) \frac{1+r}{(1-r)^3}\leq g'(0)\frac{2}{(1-r)^3}
   |\varphi(z)-z|.
\end{eqnarray*}
Now, by the above lemma, $$ |f(z)-g(z)|\leq g'(0)\frac{6}{(1-r)^3}
\sqrt{1-\varphi'(0)}=\frac{6}{(1-r)^3}\sqrt{g'(0)(g'(0)-f'(0))}.$$
\end{proof}

\begin{proof}[Proof of Theorem~\ref{characterizarioninclusionchain1}]
The fact that (1) implies (2) and (3) is just one of the implications in the Carath\'eodory kernel theorem (see,
e.g., \cite[Theorem 1 on p.\,55]{Goluzin} or \cite[Theorem~1.8 on p.\,29]{Pommerenke}). The function
$\mu_{w_\ast}$ is non-decreasing for any $w_\ast\in\Omega_0$ due to the inclusions~$\Omega_s\subset\Omega_t$,
$0\le s\le t$. Being clear that (3) implies (2), it remains to prove that (2) implies (1).

Let us fix a sequence of non-negative real numbers $(t_n)$ converging to $t_0\in[0,+\infty)$. We have to prove
that $\Omega_{t_n}$ converges to $\Omega_{t_0}$ in the sense of the kernel convergence. By the very definition
of conformal radius, there are univalent functions in the unit disk $f_n$ and $f$ such that
$f_n(\D)=\Omega_{t_n}$, $f(\D)=\Omega_{t_0}$, $f_n(0)=f(0)=w_*$, $f_n'(0)=r(\Omega _{t_n},w_{\ast })$,
$f'(0)=r(\Omega _{t_0},w_{\ast })$ for all $n$. By hypothesis, $f_n'(0)\to f'(0)$.

We have to prove that $f_n$ converges uniformly on compacta to $f$. Fix $0<r<1$. We may assume that the
sequence $(t_n)$ is monotone. Let us assume first that $(t_n)$ is non-decreasing.  By Lemma
\ref{subordination}, there is a constant $C=C(r)$ such that $$|f(z)-f_n(z)|\leq
C\sqrt{f'(0)(f'(0)-f_n'(0))}, \quad \textrm{ for all } |z|\leq r.$$ If the sequence $t_n$ is
non-increasing we obtain in a similar way that $$|f(z)-f_n(z)|\leq C\sqrt{f_n'(0)(f_n'(0)-f'(0))}, \quad
\textrm{ for all } |z|\leq r.$$ In both cases, we conclude that $f_n$ converges uniformly on compacta to
$f$. The proof of Theorem~\ref{characterizarioninclusionchain1} is now finished.
\end{proof}

Now we turn to the proof of Theorem~\ref{characterizarioninclusionchain2}. Before doing this let us
consider the following natural question: {\it is it possible to replace the set equality
$$\{\Omega_t:t\ge0\}=\{f_t(\UD):t\ge0\}$$ in Theorem~\ref{characterizarioninclusionchain2} by the stronger
requirement that $\Omega_t=f_t(\D)$ for every $t>0$?} The answer turns out to be negative, see
Example~\ref{inclusion_chain not admissible_family} below. The reason for that is explained by the
following theorem. Denote by $AC^d(X,Y)$, $X\subset \R$, $d\in[1,+\infty]$, the class of all locally
absolutely continuous functions $f:X\to Y$ such that the derivative $f'$ belongs to $L_{\rm loc}^d(X)$.

\begin{theorem}
\label{characterization-L-admissiblefamily}
Let $(\Omega _{t})_{t\geq 0}$ be a family  of simply connected
domains and $d\in[1,+\infty]$. The three statements below are equivalent:

\begin{itemize}
\item[(i)] The following two conditions are fulfilled:
\begin{itemize}
\item[AF1.] $\Omega_s\subset\Omega_t\neq \Complex$ whenever $t\ge s\ge0$.

\item[AF2.] There exists $w_{\ast }\in \Omega _{0}$ such that the function $\mu_{w_{\ast
}}(t):=r(\Omega _{t},w_{\ast })$ belongs to $AC^{d}%
\big([0,+\infty),\mathbb{R}\big).$
\end{itemize}

\item[(ii)] The family $(\Omega_t)$ satisfies condition AF1 and the following assertion holds:

\begin{itemize}
\item[AF2'.] For all $w_{\ast }\in \Omega _{0}$ the function $\mu _{w_{\ast
}}(t):=r(\Omega _{t},w_{\ast })$ is non-decreasing and belongs to $AC^{d}%
\big([0,+\infty ),\mathbb{R}\big).$
\end{itemize}

\item[(iii)] There exists a Loewner chain $(f_{t})$ of order~$d$ such that $%
f_{t}(\mathbb{D})=\Omega _{t}$ for all $t\geq 0$.
\end{itemize}
\end{theorem}

This theorem explains our motivation to introduce the following definition.

\begin{definition}
\label{def_L-adm} A family $(\Omega_t)_{t\ge0}$ of simply connected domains $\Omega_t\subset\C$ is said to be an \textit{L-admissible family of order~$d\in[1,+\infty]$} if it satisfies conditions AF1 and AF2 in Theorem~\ref{characterization-L-admissiblefamily}.
\end{definition}

\begin{example}
\label{inclusion_chain not admissible_family} Take $\gamma:[0,+\infty)\to[1,+\infty)$ a non-decreasing function
which is continuous but not absolutely continuous. For each $t\ge0$, define $\Omega_t:=\gamma(t)\D$. Obviously,
$r(\Omega _{t},0)= \gamma(t)$. The increasing family of domains $(\Omega_t)$ is an inclusion chain, but it is
not an L-admissible family and consequently there exist no Loewner chains $(f_t)$ such that
$f_{t}(\mathbb{D})=\Omega _{t}$ for all $t\geq 0$.
\end{example}

For the proof of the above theorem we will need the following lemma.

\begin{lemma}
\label{conformal_radius} Let $(f_{t})$ be a Loewner chain and $(\varphi
_{s,t})$ its evolution family. Let $z_{0}\in \mathbb{D}$ and $%
w_{0}:=f_{0}(z_{0})$. Then for each $t\geq 0$ the conformal radius~$\rho (z_{0},t)$ of~$f_{t}(\mathbb{D})$
w.r.t.~$w_{0}$ equals $|f_{0}^{\prime }(z_{0})|/\beta _{t}(z_{0})$, where
\begin{equation*}
\beta _{t}(z):=\frac{|\varphi _{0,t}^{\prime }(z)|}{1-|\varphi _{0,t}(z)|^{2}%
},\quad z\in \mathbb{D},~t\geq 0.
\end{equation*}
\end{lemma}

\begin{proof}
Let $z_{t}:=\varphi _{0,t}(z_{0})$ and
\begin{equation*}
\ell _{t}(z)=\frac{z+z_{t}}{1+\overline{z_{t}}\,z}.
\end{equation*}%
The function $g_{t}:=f_{t}\circ \ell _{t}$ maps conformally $\mathbb{D}$ onto $f_{t}(\mathbb{D})$. Moreover,
$g_{t}(0)=w_{0}$ because
\begin{equation}
f_{t}\circ \varphi _{0,t}=f_{0}.  \label{qqq}
\end{equation}%
Hence $\rho (z_{0},t)=|g_{t}^{\prime }(0)|$. By the chain rule, from~%
\eqref{qqq} it follows that $f_{t}^{\prime }(z_{t})\varphi _{0,t}^{\prime }(z_{0})=f_{0}^{\prime }(z_{0})$. Now
the proof can be completed by means of simple computations.
\end{proof}

\begin{proof}[Proof of Theorem~\ref{characterization-L-admissiblefamily}]
First of all let us prove that assertion (i) implies (iii). By the Riemann mapping theorem, for all $t\geq
0$ there is a unique univalent holomorphic function $f_t:\UD\to\C$  such that $f_{t}(\D )=\Omega _{t}$,
$f_{t}(0)=w_{\ast }$ and $f_{t}^{\prime }(0)>0.$ We will prove that the family $(f_{t})$ is a Loewner
chain of order~$d.$

Write $\varphi _{s,t}(z)=f_{t}^{-1}\circ f_{s}(z)$ for all $z\in \mathbb{D}$ and for all $0\leq s\leq t.$ It is
clear that the functions $\varphi _{s,t}$ are univalent and the family $(\varphi _{s,t})$ satisfied EF1 and EF2.
Since the function $t\mapsto b(t)=\varphi _{0,t}^{\prime }(0)=f_{0}^{\prime }(0)/f_{t}^{\prime
}(0)=f_{0}^{\prime }(0)/r(\Omega _{t},w_{\ast })$ belongs to $AC^{d}\big([0,+\infty ),\mathbb{R}\big)$ and the
map $t\mapsto a(t)=\varphi _{0,t}(0)=0$ is constant, by \cite[Proposition 2.10]{SMP}, $(\varphi _{s,t})$ is an
evolution family of order $d.$ Finally, since the functions $f_{t}$ are univalent and $f_{t}\circ \varphi
_{s,t}=f_{s}$ for all $0\leq s\leq t,$ by \cite[Lemma 3.2]{SMP}, the family $(f_{t})$ is a Loewner chain of
order $d.$

Being trivial that (ii) implies (i), it remains to prove that (iii)
implies (ii). So take a Loewner chain $(f_{t})$ of order~$d$ such that $%
f_{t}(\D)=\Omega _{t}$ for all $t\geq 0$ with associated evolution family given by $(\varphi _{s,t}).$ By the
very definition of Loewner chain, it is clear that the family $(\Omega _{t})$ satisfies AF1. Fix $w_{\ast }\in
\Omega _{0}$ and write $\mu _{w_{\ast }}(t):=r(\Omega _{t},w_{\ast }).$ Clearly, the function $\mu _{w_{\ast }}$
is non-decreasing. Take $z_{\ast }\in \D$ such that $w_{\ast }:=f_{0}(z_{\ast })$ and define $ a(t)=\varphi
_{0,t}(z_{\ast })$ and $b(t)=\varphi _{0,t}^{\prime }(z_{\ast }) $ for all $t\geq 0.$ By \cite[Proposition
2.10]{SMP}, $a$ and $b$ belong to $AC^{d}\big([0,+\infty ),\mathbb{R}\big).$ Moreover, by Lemma %
\ref{conformal_radius}, we have that $\mu _{w_{\ast }}(t)=|f_{0}^{\prime }(z_{0})|\frac{1-|a(t)|^{2}}{|b(t)|}.$
Finally, a direct computation shows that $\mu _{w_{\ast }}$ also belongs to $AC^{d}\big([0,+\infty ),\R\big).$ This completes the proof.
\end{proof}

Now we can deduce Theorem~\ref{characterizarioninclusionchain2} from Theorem~\ref{characterization-L-admissiblefamily}.

\begin{proof}[Proof of Theorem~\ref{characterizarioninclusionchain2}]
Notice that if we manage to construct a continuous  non-decreasing function $h:[0,+\infty)\to [0,+\infty)$
with $h(0)=0$ and $\lim_{t\to+\infty}h(t)=+\infty$ such that $\Omega_{h(t)}$ is an L-admissible family of
order~$d\in[1,+\infty],$ we will finish just by applying
Theorem~\ref{characterization-L-admissiblefamily}.

Denote $J:=\mu _{w_{\ast }}([0,+\infty))$ and take any monotone function $g$ mapping $[0,+\infty)$ onto $J$ that
belongs to $AC^{\infty}\big([0,+\infty),\mathbb{R}\big).$ The function~$h$ we are looking for is given by
$$h(t):=\inf\{\theta\ge0:\mu_{w_{\ast}}(\theta)=g(t)\}.$$ The proof is now finished.
\end{proof}

\section{Sequences of univalent functions}
\label{Sequences of functions}

In this section we make ready the key tool to solve our problem. We present a result on convergence of sequences
of univalent functions in the spirit of the Carath\'{e}odory kernel theorem.

When dealing with holomorphic self-mappings having a boundary fixed point  it is more convenient to
switch from the unit disk to the upper half-plane
$\UH:=\{z:\Im z>0\}$. It can be made by means of the change of variables~$w=H(\zeta)$, where
\begin{equation}\label{Cayley}H(\zeta):=i\frac{1+\zeta}{1-\zeta}\end{equation} is the so called
{\it Cayley map} which sends the unit disk~$\UD$ onto the upper
half-plane~$\UH$, with $H(0)=i$ and $H(1)=\infty$.

We will need the following classical result (see, e.\,g.\,\cite[Ch.\,IV~\S26]{Valiron}).

\begin{result}\label{ugl_pr}
Let $f:\UH\to\UH$ be a holomorphic function. Then \begin{equation}\label{dev1}f(z)=cz+\Phi(z),~\quad
z\in\UH,~\end{equation} where $$c:=\inf_{z\in\UH}\frac{\Im f(z)}{\Im z}$$ and $\Phi:\UH\to\UH$ is a holomorphic
function such that
\begin{equation}\label{dev2}\angle\lim\limits_{z\to\infty}\frac{\Phi(z)}{z}=0.\end{equation}
\end{result}
The number~$c$ in the above theorem is called the {\it angular derivative} of the function~$f$ at~$\infty$. We
will denote this number by $f'(\infty)$. A very simple but useful consequence of Theorem~D is the following
\begin{lemma}\label{lemma}
Suppose $f:\UH\to\UH$ is a holomorphic function and $A$ is a non-negative number. Then the following assertions
are equivalent:
\begin{itemize}
\item[(i)] the angular derivative of~$f$ at~$\infty$ equals~$A$;\\[0mm]

\item[(ii)] $\displaystyle\angle\lim_{z\to\infty}\frac{f(z)}{z}=A$;\\[1mm]

\item[(iii)] $\displaystyle\lim_{y\to+\infty}\frac{f(iy)}{iy}=A$;\\[1mm]

\item[(iv)] $\displaystyle\angle\lim_{z\to\infty}\frac{\Im\big(f(z)-Az \big)}{|z|}=0$;\\[1mm]

\item[(v)] $\displaystyle\lim_{y\to+\infty}\frac{\Im\big(f(iy)-iAy \big)}{y}=0$;
\end{itemize}
\end{lemma}
\begin{proof}
First suppose (i) holds. Then $\Phi(z):=f(z)-Az$ satisfies~\eqref{dev2}. It immediately follows that~(ii)--(v)
take place.

Now let us denote by $c$ the angular derivative of~$f$ at~$\infty$. Then~(ii)--(v) hold with $c$ substituted
for~$A$. At the same time whenever the function~$f$ is fixed, each of these assertions can hold for at most one
value of~$A$. This proves that (i) is implied by each of assertions~(ii)--(v).
\end{proof}

\begin{remark}\label{ex_of_ad}
Another consequence of Theorem~\ref{ugl_pr} is the following well-known assertion: {\it if~$\varphi\in{\sf
Hol}(\UD,\UD)$ and $\exists~\angle\lim_{z\to1}\varphi(z)=1$, then there exists the angular derivative
$$\varphi'(1):=\angle\lim\limits_{z\to1}\frac{\varphi(z)-1}{z-1},$$ which can be finite or infinite. If $\varphi'(1)\neq\infty$, then $\varphi'(1)>0$.}
\end{remark}

It is easy to see that if $f\in{\sf Hol}(\UH,\UH)$ and
$f'(\infty)\neq0$,  then $\angle\lim_{z\to\infty}f(z)=\infty$. The
latter means that $f$ has a boundary fixed point at $z=\infty$. If
$f'(\infty)=1$, this boundary fixed point is said to be {\it
parabolic}. Let us denote by $\mathfrak P$ the class of all
univalent functions from~${\sf Hol}(\UH,\UH)$ that have a
parabolic boundary fixed point at~$z=\infty$, $$\mathfrak
P:=\big\{f\in{\sf Hol}(\UH,\UH):f'(\infty)=1,\text{~$f$ is
univalent in~$\UH$}\big\}.$$ Similarly, by $\mathcal P$ we will
denote the class of all univalent functions $\varphi\in{\sf
Hol}(\UD,\UD)$ satisfying equalities $\varphi(1)=1$ and
$\varphi'(1)=1$ in the angular sense. Equivalently, $$\mathcal
P:=\{H^{-1}\circ f\circ H:f\in\mathfrak P\}.$$

\begin{remark}\label{Julia}
>From Theorem~\ref{ugl_pr} it follows that $\Im
\big(f(z)-z\big)\ge0$, $z\in\UH$, for all~$f\in\mathfrak P$, a
fact that will be used on numerous occasions. The equality can
happen  only for functions $f(z)=z+C$, where $C\in\Real$ is a
constant.
\end{remark}

We will also use on several occasions the following theorem due to Koebe (see, e.g., \cite[Theorem 1, \S
II.3]{Goluzin} or \cite[Proposition~2.14]{Pommerenke-II}). A curve $\Gamma:[0,1)\to\ComplexE$ is said {to land
at a point}~$p$ if there exists $\lim_{x\to1-}\Gamma(x)=p$. By $\partial_\infty E$ we will denote the boundary
of a set~$E$ in the Riemann sphere~$\ComplexE$.

\begin{result}\label{Koebe}
Let $f$ be a univalent mapping of~$\UD$ onto a domain~$D$ and $\Gamma:[0,1)\to D$ a curve in $D$ landing at some
point~$p\in\partial_\infty D$. Then the curve $f^{-1}\circ\Gamma$ lands at some point~$z_0\in\UC:=\partial\UD$.
Moreover, if $\Gamma_1:[0,1)\to D$ is a curve in $D$ landing at another point~$p_1\in\partial_\infty
D\setminus\{p\}$, then the curve $f^{-1}\circ\Gamma_1$ lands at a point different from~$z_0$.
\end{result}

Firstly we prove the following useful lemma.
\begin{lemma}\label{three functions}
Suppose that  $\varphi_j$, $j=1,2,3$, are holomorphic univalent self-mappings of~$\UD$, and
$\varphi_3=\varphi_2\circ\varphi_1$. If any two of these three functions belong to~$\mathcal P$, then so does
the third one.
\end{lemma}
\begin{proof}
Let us assume that $\varphi_2$ and $\varphi_3$ belong to~$\mathcal P$. Due to univalence of these functions, we
have $\varphi_1=\varphi_2^{-1}\circ\varphi_3$. The function $\varphi_3$ has finite angular derivative at $z=1$.
>From this one can conclude (see, e.g., \cite[p.\,80--81]{Pommerenke-II}) that $\varphi_3([0,1))$ is a
$C^1$-smooth curve landing at the point~$z=1$ and orthogonal to $\UC:=\partial\UD$ at this point. Since
$\varphi_2$ is univalent, by Theorem~\ref{Koebe} there is a point $\xi\in \partial \D$ such that $\lim_{r\to 1-}
\big(\varphi_2^{-1}\circ\varphi_3\big) (r)=\xi$.  Again the existence of  finite angular derivative of
$\varphi_2$ at the point~$z=1$ (see, e.g.,  \cite[p.\,80--81]{Pommerenke-II}) implies that  for any
$\theta\in(0,\pi/2)$ there exists $\vartheta\in(0,\pi/2)$ and $r>0$ such that the domain
$G_r(\theta):=\big\{z\in\UD:|1-z|<r,~|\arg (1-z)|<\theta\big\}$ is contained in the image~$U:=\varphi_2(G)$ of
$G=G_1(\vartheta)$. Moreover, for any $r>0$ and $\theta\in(0,\pi/2)$ there exists $\rho\in(0,1)$ such that
$\varphi_3([\rho,1))$ is contained in $G_r(\theta)$. It follows that
$\big(\varphi_2^{-1}\circ\varphi_3\big)([\rho,1))\subset G$. Hence $\xi=1$. From this we conclude that the
function~$\varphi_1=\varphi_2^{-1}\circ\varphi_3$ has at the point~$z=1$ radial limit equal to~$1$. By the
Lehto--Virtanen theorem (see, e.g., \cite[Theorem~9.3 on p.\,268]{Pommerenke}), this implies
that~$\varphi_1(1)=1$ in angular sense. According to Remark~\ref{ex_of_ad} the angular
derivative~$\varphi_1'(1)$ exists, finite or infinite. Now one can use the chain rule for angular
derivatives~\cite[Lemma~2]{SMC} to conclude that the angular derivative $\varphi'_1(1)$ actually equals~$1$.
This proves that $\varphi_2\in\mathcal P$.

Now assume that $\varphi_1$ and $\varphi_2$ belong to~$\mathcal P$. Then the curve $\varphi_2([0,1))$ lands at~$z=1$ nontangentially, i.e., $\varphi_2([0,1))\subset G_2(\vartheta)$ for some $\vartheta<\pi/2$. It follows that $\varphi_3(1)=1$ in angular sense. Therefore, we can apply the chain rule for angular derivatives to conclude that $\varphi_3\in\mathcal P$.

Finally, if $\varphi_1$ and $\varphi_3$ belong to~$\mathcal P$,
then the conclusion of the lemma follows  immediately
from~\cite[Lemma~2]{SMC}. The proof is now finished.
\end{proof}

Denote by $r_\UH(\cdot,\cdot)$ the  pseudo-hyperbolic distance in $\UH$,
$$r_\UH(z,w)=\left|\frac{z-w}{z-\bar w}\right|,\quad z,w,\in\UH.$$ The hyperbolic distance $\rho_\UH(z,w)$
between points $z,w\in\UH$ then equals $\log[(1+r)/(1-r)]$, where $r:=r_\UH(z,w)$.

\begin{remark}\label{distortion}
We will frequently use a reformulation of the well-known growth estimate for the Carath\'eodory class (see,
e.\,g., \cite[p.\,40, eq.\,(11)]{Pommerenke}). Namely, if $p$ is an analytic function in $\UH$ with $\Re
p(z)\ge0$ for all $z\in\UH$ and  $\Im p(z_0)=0$ for some $z_0\in\UH$, then
\begin{equation}\label{C}|p(z)|\le
p(z_0)\frac{1+r_\UH(z,z_0)}{1-r_\UH(z,z_0)}\end{equation} for all $z\in\UH.$
\end{remark}

Now let us prove a very technical lemma.

\begin{lemma}\label{technical}
Let $b>0$, $z\in\UH$. If $|z|>2b$ and $|\arg z-\pi/2|\le\pi/3$, then
\begin{equation}\label{tech}
\frac{1+r_\UH(z,ib)}{1-r_\UH(z,ib)}\le\frac{9|z|}{2b}.
\end{equation}
\end{lemma}
\begin{proof}
We start with the inequalities $\sqrt{1+x}\le 1+x$ for $x\ge 0$,
and $\sqrt{1+x}\ge1+x$ for $x\in[-1,0]$.  Using these inequalities
we find that if $a>b$ and $\kappa\in[0,1]$, then
$$\sqrt{a^2+b^2+2\kappa ab}\ge a+b-\frac{2ab}{a+b}(1-\kappa),$$
$$\sqrt{a^2+b^2-2\kappa ab}\le a-b+\frac{2 ab}{a-b}(1-\kappa).$$

Applying the above inequalities to $a:=|z|$ and $\kappa:=\Im z/|z|$, we conclude that
$$|z+ib|\ge|z|+b-\frac{2b|z|}{|z|+b}(1-\kappa),$$ $$|z-ib|\le|z|-b+\frac{2b|z|}{|z|-b}(1-\kappa).$$
Since $|\arg z-\pi/2|\le\pi/3$, we get that $\kappa\geq 1/2$. Therefore,
\begin{equation}\label{te}|z+ib|-|z-ib|\ge 2b\left(1-\frac{|z|^2}{|z|^2-b^2}(1-\kappa)\right)\ge
\frac{2b}3.\end{equation}
Using~\eqref{te}, we get
\begin{equation*}
\frac{1+r_\UH(z,ib)}{1-r_\UH(z,ib)}\le\frac{2|z+ib|}{|z+ib|-|z-ib|}\le\frac{3}{b}\,|z+ib|\le\frac{3}{b}\,|z|(1+b/|z|)\le\frac{9|z|}{2b},
\end{equation*}
which finishes the proof.
\end{proof}

To simplify the statement of the following two propositions, let us denote $\Natural_0:=\Natural\cup\{0\}$.

\begin{proposition}\label{l1}
Let  $(\psi_n)_{n\in\Natural}$ be a sequence of functions from
$\mathcal P$ and $(r_j)_{j\in\Natural_0}$ a sequence of  numbers
from $(-1,1)$ such that $\psi_n(r_n)=r_0$ for all $n\in\Natural$
and $\psi_n(\D)\subset\psi_m(\D)$ whenever $n<m$. If
\begin{equation}\label{geom0}
\bigcup\limits_{n\in\Natural}\psi_n(\D)=\D,
\end{equation}
then $r_n\to r_0$ and $\psi_n\to\id_{\D}$  as $n\to+\infty$.
\end{proposition}
\begin{proof}
To simplify and clarify the proof, it is better to present it in
the framework of the  upper half-plane. So assume that
$(y_n)_{n\in\Natural}$ is a sequence of positive numbers and
$(\psi_n)_{n\in\Natural}$ a
 sequence in $\mathfrak P$ such that $\psi_n(\UH)\subset\psi_m(\UH)$ whenever $n<m$. Suppose that
such that $\psi_n(iy_n)=iy_0$ for all $n\in\Natural$ and some $y_0>0$ not depending on~$n$. Suppose also that
\begin{equation}\label{geom}
\bigcup\limits_{n\in\Natural}\psi_n(\UH)=\UH.
\end{equation}
We have to show that $y_n\to y_0$ and $\psi_n\to\id_{\UH}$  as $n\to+\infty$.

Let us notice first that by Lemma~\ref{lemma}, $y_n\le y_0$ for
all~$n\in\Natural$.  Further, suppose $m>n$. Then the inclusion
$\psi_n(\UH)\subset\psi_m(\UH)$ and Lemma~\ref{three functions}
imply that $\psi_n=\psi_m\circ\omega_{n,m}$ for some
$\omega_{n,m}\in\mathfrak P$, with $\omega_{n,m}(i y_n)=i y_m$.
Again by Lemma~\ref{lemma}, we have $y_m\ge y_n$. It follows that
the sequence $(y_n)$ converges and let us denote the limit
by~$y_\infty$.

Given $f\in{\sf Hol}(\UH,\Complex)$, denote
\begin{align*}
I_f(z)&:=\left|\frac{f(z)}z-1\right|,~z\in\UH, & I_f(A)&:=\sup_{z\in A}I_f(z),~A\subset \UH,\\
J_f(z)&:=\frac{\Im\big(f(z)-z\big)}{|z|},~z\in\UH, & J_f(A)&:=\sup_{z\in A}J_f(z),~A\subset \UH.
\end{align*}
Note that $0\le J_f(z)\le I_f(z)$ for all $z\in\UH$.

 We claim that for all $\varepsilon >0$ there exist a natural number $n$ and a  positive number $R$ such that
\begin{equation}\label{cl1}
J_{\psi_m}(iy)<\varepsilon
\end{equation}
for all $m\geq n$ all all $y>R$.

Let us now prove this claim. Fix $\varepsilon\in(0,1/2)$. Let  $$A:=\{z\in\UH:|z|\ge 2y_\infty,~|\arg
z-\pi/2|\le\pi/3\}.$$ If $m\ge n$, using Remark~\ref{distortion}, we can show that for any $z\in A$,
\begin{equation}\label{eq10}
I_{\omega_{n,m}}(z)=\left|\frac{\omega_{n,m}(z)-z}{z}\right|\le
\frac{1+r_\UH(z,iy_n)}{1-r_\UH(z,iy_n)}\cdot\frac{y_m-y_n}{|z|}.
\end{equation}
Applying now Lemma~\ref{technical} for $b:=y_n$ we see that
\begin{equation}\label{eq1}
I_{\omega_{n,m}}(A)\le \frac{9}{2y_n}\,(y_m-y_n)\le\frac{9}{2}\left(\frac{y_\infty}{y_n}-1\right).
\end{equation}

Take $n_0\in\Natural$ such that $\frac{9}{2}\left(\frac{y_\infty}{y_{n_0}}-1\right)<\varepsilon$. For each $y\ge
3y_\infty$, write $$A(y):=\{z\in A:|z|\ge 2y/3\}.$$ Take $n\geq n_0$. Let us analyze how~$\omega_{n,m}$ maps the
boundary of $A(y)$. On one  hand, if $z\in A$ with $|z|=2y/3$, then by~\eqref{eq1},
$\left|\omega_{n,m}(z)-z\right|\le 2y\varepsilon /3$. On the other hand, if $z\in A$ and $\arg z=\pi/6$, then
using again~\eqref{eq1} we can deduce $|\arg\omega_{n,m}(z)-\pi/6|\le \arcsin\varepsilon$. Similarly, if $z\in
A$ and $\arg z =5\pi/6$, then $|\arg\omega_{n,m}(z)-5\pi/6|\le \arcsin \varepsilon$. This behavior of
$\omega_{n,m}$ on the boundary of $A(y)$ clearly implies that the half-line $[2y(1+\varepsilon)/3,+\infty)i$ is
contained in the set $\omega_{n,m}(A(y))$. In particular, $iy\in \omega_{n,m}(A(y))$.

Since $\psi_n=\psi_m\circ\omega_{n,m}$, given $z\in A$ we have
\begin{multline*}
J_{\psi_n}(z)=\frac{\Im\big(\psi_m(\omega_{n,m}(z))-z\big)}{|z|}=
\frac{\Im\big(\psi_m(\omega_{n,m}(z))-\omega_{n,m}(z)\big)}{|z|}+\frac{\Im\big(\omega_{n,m}(z)-z\big)}{|z|}\\=
J_{\psi_m}(\omega_{n,m}(z))\,\frac{|\omega_{n,m}(z)|}{|z|}+J_{\omega_{n,m}}(z)\ge
J_{\psi_m}(\omega_{n,m}(z))\,\frac{|\omega_{n,m}(z)|}{|z|}.
\end{multline*}
It follows that for any $y\ge 3y_\infty$ and $n\ge n_0$,
\begin{equation}\label{eq2}
J_{\psi_m}(iy)\le\frac{J_{\psi_n}\big(A(y)\big)}{1-I_{\omega_{n,m}}(A)}
\le\frac{\varepsilon}{1-\varepsilon}\le2\varepsilon.
\end{equation}
This finishes the proof of the claim.

Now let us consider the function~$\phi_n(z):=\psi_n\big(y_nL_{\alpha_n}(z/y_0)\big)$, where $L_\alpha$ is a
non-Euclidian rotation, $L_\alpha(z)=H(\alpha H^{-1}(z))$, $H(\zeta):=i(1+\zeta)/(1-\zeta)$, and
$\alpha_n:=|\psi'_n(iy_n)|/\psi'_n(iy_n)$. Since $\phi_n(iy_0)=iy_0$ and $\phi'_n(iy_0)>0$, by the
Carath\'eodory kernel  theorem (see, e.g., \cite[Theorem 1 on p.\,55]{Goluzin}) it follows from~\eqref{geom}
that $\phi_n\to\id_{\UH}$ as $n\to+\infty$. Moreover, passing to a  subsequence, we can assume that $(\alpha_n)$
converges to some~$\alpha_0$. Consequently, $\psi_n(z)\to y_0L_{\alpha_0}^{-1}(z/y_\infty)$ as $n\to+\infty$. At
the same time, according to~\eqref{cl1} composed with Lemma~\ref{lemma}, the limit function of~$(\psi_n)$ should
have a parabolic boundary fixed point at~$\tau=\infty$, i.\,e. $\alpha_0=1$ and $y_\infty=y_0$. This completes
the proof.
\end{proof}

\begin{proposition}\label{l2}
Let $(\psi_n)_{n\in\Natural}$ be a sequence of functions from $\mathcal P$ and $\{r_j\}_{j\in\Natural_0}$ a sequence of numbers from $(-1,1)$ such that
 ${r_n:=\psi_n(r_0)}$ for $n\in\Natural$. Suppose that there exists a sequence of conformal self-mappings
$(\phi_j)_{j\in\Natural_0}$ of $\D$ normalized by $\phi_j(0)=0$,
$\phi'_j(0)>0$, $j\in\Natural_0$, and a sequence
$(L_j)_{j\in\Natural_0}\subset\Moeb(\D)$ such that $\phi_n\circ
L_n\circ\psi_n=\phi_0\circ L_0$ for all $n\in\Natural$. If
$\phi_m(\D)\subset\phi_n(\D)=:G_n$ whenever $m\ge n>0$ and if
$G_0:=\phi_0(\D)$  is a connected component of the
intersection~$\cap_{n\in\Natural}G_n$, then $r_n\to r_0$ and
$\psi_n\to\id_\D$ as $n\to+\infty$.
\end{proposition}
\begin{proof} We again present the proof in the framework of self-maps of the upper half-plane. So,
let $(\psi_n)_{n\in\Natural}\subset\mathfrak P$ and ${\psi_n(iy_0)=iy_n}$, $n\in\Natural$, for some sequence
${(y_j)_{j\in\Natural_0}}$ of positive numbers. Assume that there exists a sequence of conformal mappings
$(\phi_j)_{j\in\Natural_0}$ of $\UH$ into itself normalized by $\phi_j(i)=i$, $\phi'_j(i)>0$, $j\in\Natural_0$,
and a sequence $(L_j)_{j\in\Natural_0}\subset\Moeb(\UH)$ such that:
\begin{itemize}
 \item[(1)] $\phi_n\circ L_n\circ\psi_n=\phi_0\circ L_0$
for all $n\in\Natural$,
\item[(2)]  $\phi_m(\UH)\subset\phi_n(\UH)=:G_n$ whenever $m\ge n>0$, and
\item[(3)] $G_0:=\phi_0(\UH)$ is a connected component of the interior of~$\bigcap_{n\in\Natural}G_n$.
\end{itemize}
We have to prove that $y_n\to y_0$ and $\psi_n\to\id_\UH$ as $n\to+\infty$.

According to the Carath\'eodory kernel theorem (see, e.g., \cite[Theorem 1 on p.\,55]{Goluzin}),
$\phi_n^{-1}\to\phi_0^{-1}$ uniformly on compact subsets of~$G_0$ as $n\to+\infty$. By Lemma~\ref{three
functions} we have
\begin{equation}\label{L_nLim}
L_n\circ\psi_n\to L_0,\qquad (n\to+\infty).
\end{equation}
Since $(\psi_n)$ is a normal family, each subsequence has a further subsequence that converges either to a
holomorphic function mapping $\UH$ into itself or to a constant in~$\mathbb R\cup\{\infty\}$. Let us consider
one of such limits, which we denote by $\psi$. It suffices to prove that~$\psi=\id_\UH$.

Let $m\ge n>0$. Observe that $\phi_m=\phi_n\circ\tilde \omega_{n,m}$ for some holomorphic self-mapping~$\tilde
\omega_{n,m}$ of~$\UH$. It follows that
\begin{equation}\label{por}\psi_n=\omega_{n,m}\circ \psi_m,\quad \text{with some $\omega_{n,m}\in\mathfrak
P$.}
\end{equation}
In particular, $iy_n=\omega_{n,m}(iy_m)$. Then, by Remark~\ref{Julia}, we conclude that $y_n\ge y_m$ for ${m\ge
n>0}$. Again by Remark~\ref{Julia}, $y_n\ge y_0>0$. It follows that $(y_n)$ converges to some
$y_\infty\in(0,+\infty)$. Consequently, $\psi(iy_0)=iy_\infty\in\UH$. Therefore, \eqref{L_nLim} implies
that~$\psi\in\Moeb(\UH)$. Bearing in mind that $\Re \psi(i y_0)=0$, we see that to prove~$\psi=\id_\UH$ it is
sufficient to show that $\psi\in\mathfrak P$. In its turn, the proof of this statement can be reduced with the
help of Lemma~\ref{lemma} to proving that
\begin{equation}\label{cl2}
(\forall\,\varepsilon>0)~(\exists\,R>0)~(\forall\,m\in\Natural,\,y>R)~J_{\psi_m}(iy)<\varepsilon,
\end{equation}
where $$J_f(z):=\frac{\Im\big(f(z)-z\big)}{|z|}.$$

Taking into account~\eqref{por}, we get
\begin{multline}
J_{\psi_1}(iy)=\frac{\Im\big(\omega_{1,m}(\psi_m(iy))-iy\big)}{y}=
\frac{\Im\big(\omega_{1,m}(\psi_m(iy))-\psi_m(iy)\big)}{y}+\\\frac{\Im\big(\psi_m(iy)-iy\big)}{y}=
J_{\omega_{1,m}}(\psi_{m}(iy))\cdot\frac{|\psi_m(iy)|}{y}+J_{\psi_m}(iy)\ge J_{\psi_m}(iy),
\end{multline}
because by Remark~\ref{Julia}, $J_f(z)\ge0$ for all $f\in\mathfrak
P$ and $z\in\UH$. In  view of Lemma~\ref{lemma} applied
to~$\psi_1$, this inequality implies~\eqref{cl2} and thus
completes the proof.
\end{proof}

\section{L-admissible families and chordal evolution families}
\label{main-section}

\subsection{Chordally admissible families}
Let us denote by $\Prend(D)$ the {\it Carath\'eodory boundary} of a domain~$D$, i.\,e., the set of all {\it
prime ends} of~$D$. For a Jordan domain $D$, the Carath\'eodory boundary~$\Prend(D)$ can be identified with the
topological boundary~$\partial D$.

Let $F$ be a conformal mapping of a simply connected domain~$D$. By $F^{p.e.}$ we will denote the bijection
between $\Prend(D)$ and $\Prend(F(D))$ induced by~$F$.

\begin{definition}\label{inD}
Let $G$ be a simply connected subdomain of $\UD$, $\psi:\UD\to G$ a conformal mapping, $P\in\Prend(G)$, and
$\zeta_0:=(\psi^{p.e.})^{-1}(P)$. We will say that {\it $G$ is embedded in $\UD$ conformally at the prime end
$P$} if $\zeta_0$ is a regular contact point of~the function~$\psi$, i.\,e., the following conditions hold:
\begin{itemize}
\item[(i)] $\displaystyle\exists~\angle\lim_{\zeta\to\zeta_0}\psi(\zeta):=z_0\in\UC$;
\item[(ii)] $\displaystyle\psi'(\zeta_0):=\angle\lim_{\zeta\to\zeta_0}\tfrac{\psi(\zeta)-z_0}{\zeta-\zeta_0}\neq\infty$.
\end{itemize}
\end{definition}

Definition~\ref{inD} can be extended to the case of  two arbitrary hyperbolic simply connected domains
$\Omega_1\subset \Omega_2$:

\begin{definition}\label{Conf_emb}
Let  $\Omega_1$ and  $\Omega_2$ be two hyperbolic simply connected domains in the complex plain with
$\Omega_1\subset \Omega_2$. The domain $\Omega_1$ is said to be {\it embedded in the domain~$\Omega_2$
conformally at a prime end $P_1\in\Prend(\Omega_1)$} if there is a conformal mapping~$\phi:\Omega_2\to\UD$
such that the domain~$G:=\phi(\Omega_1)$ is embedded in~$\UD$ conformally at the prime
end~$P:=\big(\phi|_{\Omega_1}\big)^{p.e.}(P_1)\in\Prend(G)$.
\end{definition}

\begin{remark}\label{norm_indep}
It is clear that Definition~\ref{inD} above does not depend on particular normalization of the conformal
mappings~$\psi$ and~$\phi$. More precisely, if $\psi_1$ and $\psi_2$ are two conformal mappings of $\UD$ onto
$G$ and $\zeta_j:=(\psi_j^{p.e.})^{-1}(P)$, $j=1,2$, then $\psi:=\psi_1$ with $\zeta_0:=\zeta_1$ satisfies
conditions~(i) and~(ii) in this definition if and only if they are fulfilled by $\psi:=\psi_2$ with
$\zeta_0:=\zeta_2$. Similar remark concerns the conformal mapping~$\phi$ in Definition~\ref{Conf_emb}.
\end{remark}

\begin{definition}\label{classC}
By $\mathcal C$ we denote the class of all univalent functions $\varphi\in{\sf Hol}(\UD,\UD)$ such
that
\begin{equation}\label{chordlaEF}
\angle\lim_{z\to1}\varphi(z)=1, \qquad \varphi'(1):=\angle\lim_{z\to1}\frac{\varphi(z)-1}{z-1}<\infty,
\end{equation}
and  let
$$
\mathcal C_0:=\{\varphi\in\mathcal C:\varphi'(1)\le1\},\\%
$$
\end{definition}

\begin{lemma}\label{conformally_embedded}
Let  $\Omega_1$ and  $\Omega_2$ be two hyperbolic simply connected domains in the complex plain with
$\Omega_1\subset \Omega_2$. There is a prime end $P_1\in\Prend(\Omega_1)$ such that the domain $\Omega_1$ is
embedded in the domain~$\Omega_2$ conformally at $P_1$ if and only if there exist two univalent functions $F_j$,
$j=1,2$, in the unit disk~$\UD$ such that
\begin{equation}\label{F}
F_j(\UD)=\Omega_j,~ j=1,2, \text{ and } F_2^{-1}\circ F_1\in{\mathcal C}.
\end{equation}
\end{lemma}
\begin{proof}
Suppose first that there exists $P_1\in\Prend(\Omega_1)$ such that $\Omega_1$ is embedded in~$\Omega_2$
conformally at~$P_1$. Take any $w_0\in\Omega_1$ and choose $F_1$ to be the conformal mapping of~$\UD$ onto
$\Omega_1$ normalized by $F_1(0)=w_0$ and $F_1^{p.e.}(1)=P_1$. Let $\phi$ be any conformal mapping
of~$\Omega_2$ onto~$\UD$ subject to the condition~$\phi(w_0)=0$. The function $\psi:=\phi\circ F_1$
maps~$\UD$ conformally into itself and satisfies conditions~(i) and~(ii) in Definition~\ref{inD} with
$$\zeta_0=(\psi^{p.e.})^{-1}\big(\big(\phi|_{\Omega_1}\big)^{p.e.}(P_1)\big)=
\big(F^{-1}_1\circ\big(\phi|_{\Omega_1}\big)^{-1}\big)^{p.e.}\big(\big(\phi|_{\Omega_1}\big)^{p.e.}(P_1)\big)=\big(F_1^{p.e.}\big)^{-1}(P_1)=1.$$
It follows that $(1/z_0)\phi\circ F_1\in\mathcal C$, where $z_0\in\UC$ stands for the angular limit
of~$\psi$ at~$\zeta_0$. Therefore, we conclude just by choosing $F_2(z):=\phi^{-1}(z_0 z)$.

Conversely, let us now assume  that there exist two univalent functions $F_j$, $j=1,2$, in the unit
disk~$\UD$ satisfying~\eqref{F}. Let $\psi:=F_2^{-1}\circ F_1$. According to  Definition~\ref{inD} and Remark~\ref{norm_indep} the domain~$G:=F^{-1}_2(\Omega_1)$ is embedded conformally in~$\UD$ at the prime
end~$P:=\psi^{p.e.}(1)$. From this we conclude in a similar way, using Definition~\ref{Conf_emb} with $\phi=F_2^{-1}$,
that $\Omega_1$ is embedded conformally in $\Omega_2$ at the prime end $P_1:=F_1^{p.e.}(1)$. The proof is
now completed.
\end{proof}

\begin{definition}\label{chordally_admissible}
An L-admissible family~$(\Omega_t)$ is said to be {\it chordally admissible} if there exists a prime
end~$P_0\in\Prend(\Omega_0)$ such that for each $t\ge0$ the domain $\Omega_0$ is embedded in the
domain~$\Omega_t$ conformally at the prime end~$P_0$.
\end{definition}

Clearly, Lemma \ref{conformally_embedded} implies

\begin{lemma}\label{adm}
An  L-admissible family $(\Omega_t)$ is chordally admissible if and only if there exist a family $(F_t)$ of
univalent holomorphic functions in the unit disk~$\UD$ such that
\begin{equation}\label{F_t}
F_t(\UD)=\Omega_t\text{ and }F_t^{-1}\circ F_0\in\mathcal C\quad \text{ for all $t\ge0$.}
\end{equation}
\end{lemma}

Let us denote by $\Spec$ the set of all locally absolutely continuous functions
${\lambda:[0,+\infty)\to\Complex}$ fulfilling the following conditions:
\begin{itemize}
\item[(i)] $\lambda'\in L^d_{loc}([0,+\infty),\Complex)$;
\item[(ii)] $\lambda(0)=0$;
\item[(iii)] $\Re\lambda(t)\ge\Re\lambda(s)$ for any~$s\ge 0$ and $t\ge s$.
\end{itemize}
Let $\SpecR$ stand for the subclass of all real-valued functions from~$\Spec$.

In~\cite[Theorem~7.1]{BCM1} it was proved that for any chordal or radial evolution family $(\varphi_{s,t})$ with
common Denjoy--Wolff point~$\tau$ we have $\varphi'_{s,t}(\tau)=\exp\big(\lambda(s)-\lambda(t)\big)$ for all
$s\ge0$, $t\ge s$ and some $\lambda\in\Spec$. Moreover, if $(\varphi_{s,t})$ is chordal, i.e.
$\tau\in\UC:=\partial{\UD}$, then $\lambda\in\SpecR$. Now let us formulate the Main Theorem.

\begin{theorem}\label{Main_Theorem}
Let $(\Omega_t)$ be an L-admissible family $(\Omega_t)$ of order ${d\in[1,+\infty]}$ and
${\tau\in\UC:=\partial{\UD}}$. Then:
\begin{itemize}
\item[1.] The following statements are equivalent:
\begin{itemize}
\item[(a)] For all $\lambda\in\SpecR$ there exists an evolution family~$(\varphi_{s,t})$ of order~$d$ and
a~Loewner chain~$(f_t)$ associated with~$(\varphi_{s,t})$ such that:
\begin{itemize}
\item[(A.1)] $\varphi_{s,t}(\tau)=\tau$ and $\varphi'_{s,t}(\tau)=\exp\big(\lambda(s)-\lambda(t)\big)$ for all
$s\ge0$ and $t\ge s$;
\item[(A.2)] $\big\{f_t(\UD):t\ge0\big\}=\big\{\Omega_t:t\ge0\big\}$.
\end{itemize}
\item[(b)] For at least one $\lambda\in\SpecR$ there exists an evolution family~$(\varphi_{s,t})$ of order~$d$ and
a~Loewner chain~$(f_t)$ associated with~$(\varphi_{s,t})$ such that the above conditions~(A.1) and~(A.2)
are fulfilled.
\item[(c)] The family $(\Omega_t)$ is chordally admissible.
\end{itemize}

\vskip3mm
\item[2.] The following statements are equivalent:
\begin{itemize}
\item[(d)] For all $\lambda\in\SpecR$ there exists an evolution family~$(\varphi_{s,t})$ of order~$d$ and
a~Loewner chain~$(f_t)$ associated with~$(\varphi_{s,t})$ such that:
\begin{itemize}
\item[(B.1)] $\varphi_{s,t}(\tau)=\tau$ and $\varphi'_{s,t}(\tau)=\exp\big(\lambda(s)-\lambda(t)\big)$ for all
$s\ge0$ and $t\ge s$;
\item[(B.2)] $f_t(\UD)=\Omega_t(\UD)$ for all~$t\ge0$.
\end{itemize}
\item[(e)] For at least one $\lambda\in\SpecR$ there exists an evolution family~$(\varphi_{s,t})$ of order~$d$ and
a~Loewner chain~$(f_t)$ associated with~$(\varphi_{s,t})$ such that  the above conditions~(B.1) and~(B.2)
are fulfilled.
\end{itemize}
\end{itemize}
\end{theorem}

\subsection{Proof of the Main Theorem}

\subsubsection{Auxiliary statements} First we prove an analogue of~\cite[Proposition~2.10]{SMP} for functions
with common parabolic fixed point on the boundary. Let us recall that by $AC^d(X,Y)$, $X\subset \R$,
$d\in[1,+\infty]$, we denote the class of all locally absolutely continuous functions $f:X\to Y$ such that
the derivative $f'$ belongs to $L_{\rm loc}^d(X)$.
\begin{proposition}\label{m_lemma}
Let $(\varphi_{s,t})_{0\le s\le t<+\infty}$ be a family of functions~from class~$\mathcal P$ and
${d\in[1,+\infty]}$. Suppose that $(\varphi_{s,t})$ satisfies conditions EF1 and EF2 in Definition~\ref{def-ev}.
If the function $t\mapsto a(t):=\varphi_{0,t}(0)$ belongs to~$AC^d_{loc}\big([0,+\infty),\Real\big)$, then
$(\varphi_{s,t})$ is an evolution family of order~$d$.
\end{proposition}
\begin{proof}
It is more convenient to work with the family $(\psi_{s,t})$ defined by $\psi_{s,t}=H\circ \varphi_{s,t}\circ
H^{-1}$, where $H$ stands for the Cayley map~\eqref{Cayley}. Then for any $s\ge0$ and $t\ge s$, we have $\psi_{s,t}(z)=z+i
p_{s,t}(z)$, where $p_{s,t}$ is some analytic function  in~$\UH$ having non-negative real part. The equality
$\varphi_{0,t}(0)=a(t)$ implies that $p_{s,t}(ib(s))=b(t)-b(s)$,  $0\le s\le t$, where
$$b(t):=(1/i)H(a(t))=\frac{1+a(t)}{1-a(t)}.$$

The key idea is to take advantage of inequality~\eqref{C} in Remark~\ref{distortion}. Apply this inequality for
$p:=p_{u,t}$ with $z_0:=ib(u)$ and $z:=\psi_{s,u}(\zeta)$, $\zeta\in\UH$, to obtain
\begin{equation}\label{psi}
|\psi_{s,t}(\zeta)-\psi_{s,u}(\zeta)|=|\psi_{u,t}(z)-z|\le\frac{1+r}{1-r}\,|b(t)-b(u)|,\quad
r:=r_{\UH}\big(ib(u),\psi_{s,u}(\zeta)\big).
\end{equation}
Since any analytic self-map of $\UH$ is a contraction for the hyperbolic metric, we have
\begin{equation}\label{r}
r_{\UH}\big(ib(u),\psi_{s,u}(\zeta)\big)=r_{\UH}\big(\psi_{s,u}(ib(s)),\psi_{s,u}(\zeta)\big)\le
r_{\UH}\big(ib(s),\zeta\big).
\end{equation}
Combining inequalities~\eqref{psi}\,and\,\eqref{r} and taking into account that $b$ is continuous, we conclude
that for any $\zeta\in\UH$ and any $T\in[0,+\infty)$ there exists a constant $C=C(\zeta,T)>0$ such that
$|\psi_{s,t}(\zeta)-\psi_{s,u}(\zeta)|\le C\,|b(t)-b(u)|$ whenever~$0\le s\le u\le t\le T$. Bearing in mind that
$|w-z|\le |H(w)-H(z)|$ for any $z,w\in\UD$, we see that this inequality proves the proposition.
\end{proof}

Let us formulate the most difficult part to prove in Theorem~\ref{Main_Theorem}.

\begin{theorem}\label{par_th}
Let $\tau\in\UC:=\partial\UD$. For any chordally admissible family~$(\Omega_t)$ there exists an evolution
family~$(\varphi_{s,t})$ of order~$+\infty$ and a Loewner chain $(f_t)$ associated with~$(\varphi_{s,t})$ such
that $\big\{f_t(\UD):t\ge0\big\}=\big\{\Omega_t:t\ge0\big\}$ and $\varphi_{s,t}(\tau)=\tau$,
$\varphi'_{s,t}(\tau)=1$ for each $s\ge0$ and $t\ge s$.
\end{theorem}
\begin{proof}
Without loss of generality one can assume that $\tau=1$. We divide the proof into the four following steps.

\Step{1}{There exists a family $(h_t)_{t\ge0}\subset{\sf Hol}(\UD,\Complex)$ that satisfies the following
conditions:}
\begin{itemize}
\item[(i)] $h_t(\UD)=\Omega_t$ for all $t\ge0$;
\item[(ii)] the mappings $\phi_{s,t}:=h_t^{-1}\circ h_s$, $0\le s\le t$, belong to the class $\mathcal P$,
\item[(iii)] $\phi_{0,t}(0)\in(-1,1)$ for all $t\ge 0$.
\end{itemize}

This part is easy. By hypothesis, $(\Omega_t)$ is a chordally admissible family. Hence by Lemma~\ref{adm}, there
exists a family $(F_t)$ of univalent functions in the unit disk~$\UD$ satisfying~\eqref{F_t}. Write
$\Phi_t=F_t^{-1}\circ F_0$.

Now let us make use of the following obvious assertion: for any function $\varphi\in\mathcal C$ there exists
$m[\varphi]\in\Moeb(\UD)$ such that $m[\varphi]\circ f\in\mathcal P$ and $\Im m[\varphi](\varphi(0))=0$. This
M\"obius transformation can be constructed as a composition of three transformations:
$$z\mapsto\frac{z-\varphi(0)}{1-\overline{\varphi(0)}z},~~z\mapsto e^{i\theta}z,~~\text{and}~z\mapsto
\frac{z-a}{1-az},$$ for appropriately chosen $\theta\in\Real$ and $a\in(-1,1)$. Setting $\varphi:=\Phi_{t}$ and
taking into account that $\Phi_{0}=\id_\UD$, we conclude that the functions
$h_t:=F_{t}\circ\big(m[\Phi_{t}]\big)^{-1}$, $t\ge0$, form a family satisfying condition~(i),\,(ii) and (iii),
for all~$t\ge0$.

In order to establish~(ii) for all $s\in[0,t]$ it suffices to notice that
$\phi_{0,t}=\phi_{s,t}\circ\phi_{0,s}$, and apply Lemma~\ref{three functions} to make sure that~$\phi_{s,t}$ is
also in~$\mathcal P$.

\Step{2}{The mapping $\mathcal Q:t\mapsto \phi_{0,t}=h_t^{-1}\circ h_0$ from $[0,+\infty)$ into ${\sf
Hol}(\UD,\UD)$ is continuous.}

Denote $x(t):=\phi_{0,t}(0)\in(-1,1)$. Let $t_n\in[0,+\infty)$ be such that the sequence $(t_n)$ converges to a
point $t_0$. The assertion of Step~2 will be proved if we show that~$\phi_{0,t_n}$ converges to $\phi_{0,t_0}$,
as $n\to+\infty$, uniformly on compact subsets of the unit disk. We can make even further simplification without
loss of generality. Namely, we can assume that $(t_n)$ is either increasing or decreasing. Let us start with the
first case.

\step{2a}{If $t_n$ is increasing, then $\phi_{0,t_n}\to\phi_{0,t_0}$ as $n\to+\infty$.}

For $n\in\Natural_0$ we denote $\psi_n:=\phi_{t_n,t_0}$ and $r_n:=x(t_n)$. Note that $\psi_n\in\mathcal P$.
Further, observe that $\psi_n(\UH)=h_{t_0}^{-1}\big(h_{t_n}(\UD)\big)$. By this reason,
$\psi_n(\UH)\subset\psi_m(\UH)$ whenever $0\neq n<m$. By assertion~(ii) in Remark~\ref{inclusion_chain_Remark} we have $\cup_{n\in\Natural}\psi_n(\UH)=\UH$. Summarizing the above argument we see that $(\psi_n)$
satisfies the hypothesis of Proposition~\ref{l1}. As a result we conclude that $\psi_n\to\psi_0=\id_\UH$ as
$n\to+\infty$.  Since $\psi_n\circ\phi_{0,t_n}=\phi_{0,t_0}$ is independent of~$n$, this implies that
$\phi_{0,t_n}$ converges to $\phi_{0,t_0}$.

Therefore, the proof of~Step~2a is now finished and we proceed with considering the case of decreasing
sequence~$(t_n)$.

\step{2b}{If $t_n$ is decreasing, then $\phi_{0,t_n}\to\phi_{0,t_0}$ as $n\to+\infty$.}

As in the previous case, we denote $r_n:=x(t_n)$, $n\in\Natural_0$. Set $\psi_n:=\phi_{t_0,t_n}\in\mathcal P$.
Then we have $\psi_n(r_0)=r_n$ for all $n\in\Natural$. Set $T:=t_1$. Consider mappings $\phi_{t_n,T}$,
$n\in\Natural_0$, and let $L_n$ be the unique function in $\Moeb(\UD)$ such that the functions
$\phi_n:=\phi_{t_n,T}\circ L_n^{-1}$, $n\in\Natural_0$, satisfy the normalization~$\phi_n(0)=0$ and
$\phi'_n(0)>0$, required in Proposition~\ref{l2}. Then we have
$$\phi_n\circ L_n\circ\psi_n=\phi_{t_n,T}\circ\phi_{t_0,t_n}=\phi_{t_0,T}=\phi_0\circ L_0.$$ Furthermore, the
domains $G_n:=\phi_n(\UD)=\phi_{t_n,T}(\UD)=h_T^{-1}\big(h_{t_n}(\UD)\big)$ form, starting from $n=1$, a
contracting sequence of domains, with $G_0$ being, according to assertion~(i) in Remark~\ref{inclusion_chain_Remark}, a connected component of the interior of~$\cup_{n\in\Natural}G_n$.

Hence, we can apply Proposition~\ref{l2} to conclude that $\psi_n\to\id_H$ as $n\to+\infty$. As a consequence,
$\phi_{0,t_n}$ converges to $\phi_{0,t_0}$, which is what we had to show in order to complete the proof of
Step~2b.

\Step3{Let $x(t):=\phi_{0,t}(0)$. There exists a non-decreasing function $\theta:[0,+\infty)\to[0,+\infty)$ with
$\theta(0)=0$ such that $t\mapsto x(\theta(t))$ is a smooth function on $[0,+\infty)$ and
\begin{equation}\label{Omega_range} \big\{h_{\theta(t)}:t\ge0\big\}=\big\{h_t:t\ge0\big\}.
\end{equation}}

Let us first of all mention that $x(0)=0$ and that $x(t)=\phi_{s,t}(x(s))$ for any $s\ge0$ and $t\ge s$. Passing
this equation to the upper half-plane and using that the function $x$ is real valued, from Remark~\ref{Julia} it
follows that the function $t\mapsto x(t)$, ${t\in[0,+\infty)}$, is non-decreasing. Furthermore, by Step~2 this
function is continuous.

In particular, it follows that $J:=x\big([0,+\infty)\big)$ is an interval of the form $[0,a)$ or $[0,a]$ for
some $a\in(0,1]$. Take a function $\chi:[0,+\infty)\to J$ smooth and non-decreasing. Now for each $t\ge0$ let us
set
$$\theta(t):=\inf \{\theta\ge0:x(\theta)=\chi(t)\}.$$ Since $x(0)=0$, immediately we deduce $\theta(0)=0$.
Since the functions $\theta\mapsto x(\theta)$ and $t\mapsto \chi(t)$ have the same range on~$[0,+\infty)$, the
function $t\mapsto \theta(t)$ is well defined on~$[0,+\infty)$. Furthermore, since $\theta\mapsto x(\theta)$ is
continuous it follows that $x(\theta(t))=\chi(t)$ for all $t\ge0$, which means that $t\mapsto x(\theta(t))$ is
smooth. Finally, $t\mapsto\theta(t)$ is non-decreasing because both functions $\theta\mapsto x(\theta)$ and
$t\mapsto \chi(t)$ are non-decreasing.

It remains to show that \eqref{Omega_range} holds. To this end denote
$\Theta:=\theta\big([0,+\infty)\big)$ and let us again apply Remark~\ref{Julia} and equality
$x(t)=\phi_{s,t}(x(s))$ to conclude that $x(s)=x(t)$ holds for some $t,s\ge0$ if and only if
$\phi_{s,t}=\id_\UD$, which is in its turn equivalent to~$h_t=h_s$. It follows that for any two sets $A_1,
A_2\in[0,+\infty)$ the ranges of $t\mapsto h_t$ on $A_1$ and on $A_2$ coincide if and only if the same is
true for $t\mapsto x(t)$. By construction, the range of $t\mapsto x(t)$ on $\Theta$ is the same as on the
whole semiaxis~$[0,+\infty)$. This implies~\eqref{Omega_range} and finishes the proof of Step~3.

\Step4{The family $(\varphi_{s,t})$ defined by $\varphi_{s,t}=\phi_{\theta(s),\theta(t)}$, $0\le s\le t$, is an
evolution family of order~$+\infty$ and the family $(f_t)$ defined by $f_t:=h_{\theta(t)}$ is a Loewner chain
associated with the evolution family~$(\varphi_{s,t})$.}

Clearly, the family $(\varphi_{s,t})$ satisfies conditions EF1 and EF2 in Definition~\ref{def-ev}. Moreover, by
the construction, the function $t\mapsto \varphi_{0,t}(0)$ is real-valued and smooth. Hence, by
Proposition~\ref{m_lemma}, $(\varphi_{s,t})$ is an evolution family of order~$d=+\infty$. Note that
$f_t\circ\varphi_{s,t}=f_s$ for all $s\ge0$ and $t\ge s$. Then by~\cite[Lemma~3.2]{SMP}, $(f_t)$ is a Loewner
chain of order~$d=\infty$.

The proof of Theorem~\ref{par_th} is now completed.
\end{proof}

The following statement is a direct consequence of~\cite[Lemma~2.8]{SMP}.

\begin{lemma}\label{Parabolic_to_general SpFunction}
Let $\tau\in\UC$ and $\lambda\in\SpecR$. Let~$(\psi_{s,t})$ be a family of holomorphic self-maps of~$\UD$, and
\begin{equation}\label{arbitrary_spectral_function}
\varphi_{s,t}:=h_t\circ\psi_{s,t}\circ h^{-1}_s,\quad t\ge s\ge0,\quad h_t(z):=\tau\frac{z-\tau
a(t)}{\tau-a(t)z},~a(t):=\frac{e^{\lambda(t)}-1}{e^{\lambda(t)}+1}.
\end{equation}
Then the following two statements are equivalent:
\begin{itemize}
\item[(i)] the family $(\varphi_{s,t})$ is an evolution family of order~$d$ with $\varphi_{s,t}(\tau)=\tau$ and
$\varphi'_{s,t}(\tau)=\exp\big(\lambda(t)-\lambda(s)\big)$ for all $s\ge 0$ and $t\ge s$;

\item[(ii)] the family $(\psi_{s,t})$ is an evolution family of order~$d$ with $\psi_{s,t}(\tau)=\tau$ and
$\psi'_{s,t}(\tau)=1$  for all $s\ge 0$ and $t\ge s$.
\end{itemize}
\end{lemma}

Now we are ready to prove Theorem~\ref{Main_Theorem}.

\subsubsection{Proof of Statement~1 of Theorem~\ref{Main_Theorem}} Let us use Lemma~\ref{Parabolic_to_general
SpFunction} in order to deduce Statement~1 of Theorem~\ref{Main_Theorem} from Theorem~\ref{par_th}. Clearly, in
view of Theorem~\ref{par_th}, (c)~implies~(b). Let us show that (c) also implies~(a). To this end take an
arbitrary~$\lambda\in\SpecR$ and define~$(\varphi_{s,t})$ by~\eqref{arbitrary_spectral_function} and $(f_t)$ by
$f_t:=g_t\circ h_t^{-1}$. Then $(\varphi_{s,t})$ and $(f_t)$  satisfy conditions~(A.1) and~(A.2). Moreover,
by~\cite[Lemma~3.2]{SMP}, $(f_t)$ is a Loewner chain of order~$d$ associated with the evolution
family~$(\varphi_{s,t})$.  This proves that~(c) implies~(a).

It remains to show that (b) implies (c). In view of Lemma~\ref{adm} it suffices to verify that to each $t\ge0$
one can assign a conformal mapping $F_t$ of $\UD$ onto $\Omega_t$ in such a way that $F_t^{-1}\circ
F_0\in\mathcal C$ for any $t\ge0$. According to (A.2) for each $t\ge0$ there exists $\theta\ge0$ such that
$f_\theta(\UD)=\Omega_t$. Note that $f_0(\UD)=\Omega_0$ because there can be no two different minimal elements
in a linearly ordered set. Hence we can assume that if $t=0$, then $\theta$ also vanishes.

Now take $F_t:=f_\theta$. Then $F_t^{-1}\circ F_0=\varphi_{0,\theta}\in\mathcal C$ by~(A.1). This completes the
proof of Statement~1. \proofbox\vskip5mm

\subsubsection{Proof of Statement~2 of Theorem~\ref{Main_Theorem}} This statement basically follows from
Lemma~\ref{Parabolic_to_general SpFunction}. Suppose (e) is true, i.e., there exists $\lambda_0\in\SpecR$ such
that (B) holds for $\lambda:=\lambda_0$ with some evolution family $(\varphi^0_{s,t})$ and some associated
Loewner chain~$(f_t^0)$.

Let us prove that (d) holds. Define a family $(\psi_{s,t})\subset{\sf Hol}(\UD,\UD)$ by the following formula
\begin{equation*}
\psi_{s,t}:=(h^0_t)^{-1}\circ\varphi^0_{s,t}\circ h^0_s,\quad t\ge s\ge0,\quad h^0_t(z):=\tau\frac{z-\tau
a_0(t)}{\tau-a_0(t)z},~a_0(t):=\frac{e^{\lambda_0(t)}-1}{e^{\lambda_0(t)}+1},
\end{equation*}
where $h_t^0$ coincides with $h_t$ given in~\eqref{arbitrary_spectral_function} when~$\lambda=\lambda_0$. We
apply Lemma~\ref{Parabolic_to_general SpFunction} for $\lambda:=\lambda_0$ to conclude that the family
$(\psi_{s,t})$ is an evolution family, satisfying condition~(ii) in this lemma.

Take now any arbitrary $\lambda\in\SpecR$ and use Lemma~\ref{Parabolic_to_general SpFunction} again with
this~$\lambda$ to construct an evolution family~$(\varphi_{s,t})$ satisfying condition~(i) in this lemma, which
implies condition~(B.1). Define $(f_t)$ by $f_t:=f^0_t\circ h^0_t\circ h_t^{-1}$, $t\ge 0$. Note that
$f_t(\UD)=f_t^0(\UD)$ for all~$t\ge0$. It follows that condition~(B.2) is also satisfied. To complete the proof
it remains to make sure that $f_t\circ\varphi_{s,t}=f_s$ for all $s\ge0$ and $t\ge s$ and
apply~\cite[Lemma~3.2]{SMP} to conclude that $(f_t)$ is a Loewner chain of order~$d$ associated with the
evolution family~$(\varphi_{s,t})$. \proofbox

\section{L-admissible families and Goryainov--Ba evolution families}
\label{Goryainov-Ba}

\subsection{Definitions and results}

Originally the chordal variant of Loewner Theory was developed for the class~$\mathcal H$ of univalent
self-mappings~$G$ of the upper half-plane~$\UH:=\{z:\Im z>0\}$ satisfying the condition
\begin{equation*}
\lim_{\substack{z\to\infty,\\z\in\UH}}\big(G(z)-z\big)=0.
\end{equation*}

The parametric representation was given for a dense subclass of~$\mathcal H$, all elements of which
satisfy quite strong regularity condition at the point of infinity, namely $\UH\setminus G(\UH)$ is a
bounded set, so that $G$ has meromorphic continuation to a neighbourhood of~$\infty$ with real
coefficients in its Laurent expansion. The role of Loewner equation~\eqref{L_ODE_intro} in this case is
played by the so-called {\it chordal Loewner equation}
\begin{equation}\label{chLE}
\frac{dw}{dt}=\frac{1}{\lambda(t)-w},\quad t\ge0;~~w|_{t=0}=z,~z\in\UH,
\end{equation}
where $\lambda:[0,+\infty)\to\Real$ is a (piece-wise) continuous driving term.

Here we cite some papers devoted to this topic without attempts to give the complete bibliography
\cite{Popova1949, Popova1954, Kufarev_etal, AleksSob, Sobolev1970}. Mainly the theory was developed as a
tool for obtaining functional estimates in the class~$\mathcal H$.

Much later, the chordal Loewner evolution attracted interest of the wide mathematical society due to the famous
paper by Schramm~\cite{Schramm}, where he introduced and studied stochastic version of Loewner evolution (SLE),
based on the simplest variant of radial and chordal Loewner equations. This development led to deep results in
mathematical theory of 2D lattice models playing important role in Statistical Physics~see, e.g.,
[\mcite{Lawler}--\mcite{Lawler-Schramm-Werner}].

Here we consider a more general and, in some sense, more natural approach than the one developed
in~\cite{Aleks1983, AleksST, Goryainov-Ba}, see also~\cite{AleksSTSob}. Very close results were independently
obtained later by Bauer~\cite{Bauer}, motivated by a problem connected to the Cauchy transform of probability
measures.

For any function $G\in\mathcal H$ such that $\UH\setminus G(\UH)$ is bounded, we have
\begin{equation}\label{G_expan}
G(z)=z-c/z+\gamma(z),\quad c\ge0
\end{equation}
where $\gamma(z)=o(1/z)$, i.e.,
\begin{equation}\label{gamma_prop}
\lim_{\substack{z\to\infty,\\z\in\UH}}z\gamma(z)=0.
\end{equation}
The idea in~\cite{Aleks1983, AleksST, Goryainov-Ba} was to interpret expansion~\eqref{G_expan} in angular
sense, i.e., to replace~\eqref{gamma_prop} with the weaker condition
\begin{equation}\label{gamma_prop_angular}
\angle\lim_{z\to\infty}z\gamma(z)=0.
\end{equation}

This leads to the class of functions~$\mathfrak P_0$ characterized in the following way.

\begin{definition}[\cite{Bauer}]\label{class_mathfrak_P_0}
By $\mathfrak P_0$ we denote the class of all univalent function $G:\UH\to\UH$ satisfying the following
condition: there exists $C>0$ such that ${|G(z)-z|\le C/\Im z}$ for all ${z\in\UH}$. Let us denote by $\ell(G)$
the minimal value of~$C$ for which the above estimate holds.
\end{definition}

\begin{remark}\label{GorBaBauer_Remark}
Basic facts about the class~$\mathfrak P_0$ can be found in~\cite{Aleks1983, Goryainov-Ba, Bauer}. In
particular, the class~$\mathfrak P_0$ is a semigroup with respect to the operation of composition and the
functional~$\ell$ turns out to be additive, i.e., $\ell(G_1\circ G_2)=\ell(G_1)+\ell(G_2)$,
$G_j\in\mathfrak P_0$, $j=1,2$. Moreover, for any $G\in\mathfrak P_0$ and for any $\varepsilon>0$, we have
\begin{equation*}
    \ell(G)=\lim_{\substack{z\to\infty,\\\Im (z)>\varepsilon}}z\big(z-G(z)\big).
\end{equation*}
Another useful fact is that the class~$\mathfrak P_0$ coincides with the class of all univalent functions~$F$
in~$\Hol$ that admit the following integral representation:
\begin{equation}\label{P_0_repesent}
F(z)=z+\int_{\Real}\frac{d\mu(x)}{x-z},
\end{equation}
where $\mu$ is a finite positive Borel measure on~$\Real$.
\end{remark}

In~\cite{Aleks1983, AleksST} and later in~\cite{Goryainov-Ba} there was given a generalization of the
original chordal Loewner equation~\eqref{chLE} for evolution families in the class $\mathfrak P_0$, with
the functional $\ell$ playing the role of the governing parameter.

The fact that the class~$\mathfrak P_0$ can be characterized by expansion~\eqref{G_expan} with the
residual term $\gamma$ subject to condition~\eqref{gamma_prop_angular}, follows directly
from~\cite[Proposition~7 and Remark~2]{Aleks1983}.

Now we can switch to the unit disk~$\UD$ as a reference domain by means of the Cayley map~$H$, see
formula~\eqref{Cayley}. The class $\mathcal P_0:=\{\varphi=H^{-1}\circ G\circ H: G\in\mathfrak P_0\}$ can
be also defined as the subclass of~$\mathcal P$ that consists of all functions $\varphi$ having finite
angular derivatives at the point~$\tau=1$ up to the third order with $\varphi''(1)=0$. In other words,
$\varphi\in\mathcal P$ belongs to~$\mathcal P_0$ if and only if
\begin{equation}\label{P_0}\varphi(z)=1+(z-1)-\frac{c\,(z-1)^3}{4} + \gamma(z)\end{equation} for some
$c\in\Complex$ and $\gamma\in{\sf Hol}(\UD,\Complex)$ such that
$$\angle\lim_{z\to1}\dfrac{\gamma(z)}{(z-1)^3}=0.$$

Suppressing the language, we define $\ell:\mathcal P_0\to[0,+\infty)$ by $\ell(\varphi):=\ell(H\circ\varphi\circ
H^{-1})$. In view of Remark~\ref{GorBaBauer_Remark}, it is easy to check by elementary computations that  the
coefficient~$c$ in expansion~\eqref{P_0} equals~$\ell(\varphi)$. It is interesting to mention that the
class~$\mathcal P_0$ admits a rigidity property, the so-called {\it Burns--Krantz theorem},
see~\cite{Burns-Krantz}, stating that if $\varphi\in\mathcal P_0$ and $\ell(\varphi)=0$, then
$\varphi=\id_{\UD}$.

Now we introduce the notion of Goryainov--Ba evolution families, which coincides essentially (but not literally)
with those studied by Goryainov, Ba and Bauer.

\begin{definition}\label{def_GorBa}
A {\it Goryainov--Ba evolution family} is an evolution family $(\varphi_{s,t})$ contained in $\mathcal P_0$ and
such that the function $t\mapsto\ell(\varphi_{0,t})$ is locally absolutely continuous on $[0,+\infty)$.
\end{definition}

Under conditions EF1 and EF2 from Definition~\ref{def-ev} of evolution families, local absolute continuity
of~$t\mapsto\ell(\varphi_{0,t})$ implies condition~EF3. This is the content of the following simple proposition.

\begin{proposition}\label{GorBaLemma}
Suppose that a family~$(\varphi_{s,t})$, $s\ge0$, $t\ge s$, is contained in~$\mathcal P_0$ and satisfies
conditions EF1 and EF2 in Definition~\ref{def-ev}. Let $v(t):=\ell(\varphi_{0,t})$, $t\ge0$. If ${v\in
AC^d\big([0,+\infty),[0,+\infty)\big)}$, then $(\varphi_{s,t})$ is a Goryainov--Ba evolution family of
order~$d$.
\end{proposition}
\begin{proof}
We have only to prove that $(\varphi_{s,t})$ satisfies condition EF3, what will be done in the framework
of the upper half-plane.

Let us consider the functions $$\Phi_{s,t}:=H\circ \varphi_{s,t}\circ H^{-1}.$$ Note that
$\Phi_{s,t}\in\mathfrak P_0$. It follows that whenever $0\le s\le u\le t$, we have
\begin{equation}
\big|\Phi_{s,t}(z)-\Phi_{s,u}(z)\big|= \big|\Phi_{u,t}(\zeta)-\zeta\big|\le \frac{v(t)-v(u)}{\Im \zeta}\le
\frac{v(t)-v(u)}{\Im z},
\end{equation}
where $\zeta:=\Phi_{s,u}(z)$.

Therefore condition~EF3 holds with $k_{z,T}:=v'|_{[0,T]}/\Im z$. This completes the proof.
\end{proof}

Our main result in this section is a complete characterization of the geometry of expanding systems of domains
which can be obtained as the set of image domains of some Loewner chain associated with a Goryainov--Ba
evolution family.

For such a characterization, we introduce the following subclass of the class~$\mathcal C$ (see
Definition~\ref{classC}):

\begin{multline*}\classCC:=\big\{\varphi\in\mathcal
C:\text{$\varphi^{(j)}(1)\neq\infty$ exists in angular sense for $j=1,2,3$,}\\ \text{and
$\Re\varphi''(1)=\varphi'(1)\big(\varphi'(1)-1\big)/2$}\big\}.
\end{multline*}

Note that the condition $\Re\varphi''(1)=\varphi'(1)\big(\varphi'(1)-1\big)/2$ in the above definition is
to ensure that the coefficient $b$ in the expansion $$ \big(H\circ \varphi\circ H^{-1}\big)(z)=a z+b+\frac
cz+\gamma(z),\quad z\in\UH,~~\angle\lim_{z\to\infty}z\gamma(z)=0, $$ is a real number for
any~$\varphi\in\classCC$.

\begin{theorem}\label{GorBaTheorem}
Let $(\Omega_t)$ be an L-admissible family. Then the following two statements are equivalent:
\begin{itemize}
\item[(i)] there exists a Goryainov--Ba evolution family~$(\varphi_{s,t})$ and
a~Loewner chain~$(f_t)$ associated with~$(\varphi_{s,t})$ such that
\begin{equation*}\label{GorBa_f_t_Omega_t}
\big\{f_t(\UD):t\ge0\big\}=\big\{\Omega_t:t\ge0\big\};
\end{equation*}

\item[(ii)] there exist a family $(F_t)$
of univalent functions in the unit disk~$\UD$ such that
\begin{equation*}\label{GorBa_F_t}
F_t(\UD)=\Omega_t\ \text{ and }\ F_t^{-1}\circ F_0\in\mathcal \classCC\quad \text{ for all $t\ge0$.}
\end{equation*}
\end{itemize}
\end{theorem}

\begin{remark}
The evolution family~$(\varphi_{s,t})$ and the Loewner chain~$(f_t)$ we construct in the proof of the
``$(ii)\Rightarrow(i)$"-part of the above theorem, are of order~$d=\infty$.
\end{remark}

\subsection{The class $\mathfrak P_0$}
In order to prove Theorem~\ref{GorBaTheorem} we need to establish some new results about $\mathfrak P_0$. As
usual, for any $b>0$, we write $\UH_b:=\{w\in\UH:\Im w>b\}$.

\begin{lemma}\label{H_epsilon}
For each function $F\in\mathfrak P_0$ there exists $b>0$ such that $\UH_b\subset F(\UH)$.
\end{lemma}
\begin{proof}
We will prove a little bit more, namely that $\UH_b\subset F(\UH_1)$. Suppose on the contrary
that~$\UH_b\setminus F(\UH_1)\neq\emptyset$ for any $b>0$. According to Remark~\ref{Julia}, $$\UH_b\cap
F(\UH_1)\supset\UH_b\cap F(\UH_b)=F(\UH_b)\neq\emptyset$$ for any $b\ge1$. It follows that
\begin{equation}\label{UH_b}
\UH_b\cap\partial F(\UH_1)\neq\emptyset,\quad\text{for all~$b\ge1$}.
\end{equation}
Using \cite[Theorem 3.1]{ClusterSets} it is easy to see that  $\partial F(\UH_1)=\overline{F(\mathbb L)}$, where
$\mathbb{L}:=\{x+i:x\in\Real\}$. Consequently, \eqref{UH_b} implies that $\UH_b\cap F(\mathbb L)\neq\emptyset$
for each $b\ge1$. It follows that $\Im F$ is not bounded on~$\mathbb L$.

At the same time, $|F(z)-z|\le 1/\Im z=1$ for all~$z\in\mathbb L$ because~$F\in\mathfrak P_0$. This
contradiction proves the lemma.
\end{proof}

By~\cite[Lemma 1]{Goryainov-Ba}, a function $F\in\mathfrak P_0$ if and only if the following two conditions are
satisfied:
\begin{align}
\label{limy}F(iy)-iy\to0\quad\text{as}\quad y\to+\infty,~y>0, \\
\label{supy}\sup_{y>0} y\,\big(\Im F(iy)-y\big)<+\infty.
\end{align}

Using this fact, we easily deduce the following

\begin{lemma}\label{mlemma}
Let $F:\UH\to\UH$ be a univalent holomorphic  function. Then the following statements are equivalent:
\begin{itemize}
 \item[(i)] $F\in\mathfrak P_0$;
\item[(ii)] for each $c\ge0$ the function $z\mapsto F(z+ic)-ic$ belongs to the class~$\mathfrak P_0$;
\item[(iii)] there exists $c\ge0$ such that the function $z\mapsto F(z+ic)-ic$ belongs to the class~$\mathfrak P_0$.
\end{itemize}
\end{lemma}

\begin{proposition}\label{P_0_composition}
Let $\varphi_j\in\mathcal P$, $j=1,2,3$. If $\varphi_3=\varphi_2\circ\varphi_1$ and $\varphi_3\in\mathcal P_0$,
then there exists $h\in\Moeb(\UD)\cap\mathcal P$ such that $h\circ \varphi_1$ and $\varphi_2\circ h^{-1}$ belong
to $\mathcal P_0$.
\end{proposition}
\begin{proof}
To make the proof easier we reformulate the statement of the proposition in the framework of the upper
half-plane. So let $\Phi_3\in\mathfrak P_0$, $\Phi_1,\Phi_2\in\mathfrak P$ and $\Phi_3=\Phi_2\circ
\Phi_1$. We have to prove that there exists a constant $A\in\Real$ such that the functions ${z\mapsto
\Phi_1(z)+A}$ and ${z\mapsto \Phi_2(z-A)}$ belong $\mathfrak P_0$.

Since $\Phi_j\in\mathcal P$, we can write $$\Phi_j(z)=z+\tilde\Phi_j(z),\quad j=1,2,3,$$ where
$\tilde\Phi_j:\UH\to\overline\UH$ are  holomorphic functions satisfying%
$$ \angle\lim_{z\to\infty}\frac{\tilde\Phi_j(z)}{z}=0.$$ This means that the functions~$\tilde\Phi_j$ has
vanishing angular derivative at $\infty$. Using Lemma~\ref{lemma} one can easily conclude that the same holds
for the function~$\Psi_2:=\tilde\Phi_2\circ\Phi_1$. Therefore, with the notation $\Psi_3:=\tilde\Phi_3$,
$\Psi_1:=\tilde\Phi_1$, we have:
\begin{equation}\label{tildePsi}
\angle\lim_{z\to\infty}\frac{\Psi_j(z)}{z}=0,\quad j=1,2,3.
\end{equation}

The relation~$\Phi_3=\Phi_2\circ \Phi_1$ can be rewritten now in a very simple form
\begin{equation}\label{P_0_comp}
\Psi_3=\Psi_2+\Psi_1.
\end{equation}

It is known (see, e.g., \cite[Vol.\,2, p.\,7--9]{Akh} or \cite[p.18-20]{Dono}) that every holomorphic
function~$\Psi:\UH\to\overline\UH$ can be represented in the following way:
\begin{equation}\label{represent}
\Phi(z)=\beta z+\alpha+\int_\Real\left(\frac1{x-z}-\frac x{1+x^2}\right)d\mu(x),
\end{equation}
where $\beta\ge0,~\alpha\in\Real$,   and $\mu$ is a positive Borel measure on~$\Real$ such that $$\int_\Real
d\mu(x)/(1+x^2)<+\infty.$$ All the parameters in representation~\eqref{represent}, $\alpha$, $\beta$, and $\mu$
are defined by the function~$\Phi$ in a unique way.

Denote by $V(z,x)$ the integrand in~\eqref{represent}. Obviously,
$V(z,x)/z\to 0$ as $z\to\infty$ point-wise w.r.t. $x\in\Real$.
Moreover, it is easy to check that
$$\big|V(z,x)/z\big|\le\frac{1}{1+x^2} \left[\frac1{|z|\,\Im z}+\frac{\Re z}{\Im z}+1\right]$$ for all $x\in\Real$ and $z\in\UH$. Therefore, according to the Lebesgue dominated convergence theorem, $(1/z)\int_\Real V(z,x)d\mu(x)\to0$ as $z$ tends to~$\infty$ within any Stolz angle $$\Delta_\varepsilon:=\big\{z:\Im z>\varepsilon|z|\big\}, \quad\varepsilon>0.$$ Hence we conclude that the parameter $\beta$ in~\eqref{represent} equals the angular derivative of~$\Phi$ at~$\infty$.

Now let us write representation~\eqref{represent} for the functions $\Psi_j$. According to~\eqref{tildePsi}, it
will have the following form:
\begin{equation}\label{Psi_j}
\Psi_j(z)=\alpha_j+\int_\Real\left(\frac1{x-z}-\frac x{1+x^2}\right)d\mu_j(x).
\end{equation}
On the other hand, $\Phi_3\in\mathfrak P_0$. Consequently, by Remark~\ref{GorBaBauer_Remark},
$$\Psi_3(z)=\int_{\Real}\frac{d\mu_0(x)}{x-z}$$ for some finite measure Borel measure on~$\Real$. It
follows that
\begin{equation}\label{Psi_3}
\mu_3=\mu_0,\quad \alpha_3=\int_\Real\frac{x}{1+x^2}\,d\mu_0(x).
\end{equation}
Furthermore, according to the uniqueness of~$\alpha_j$ and~$\mu_j$ in representation~\eqref{Psi_j},
equality~\eqref{P_0_comp} implies that
\begin{equation}\label{mu_3}
\mu_3=\mu_2+\mu_1,\quad \alpha_3=\alpha_2+\alpha_1.
\end{equation}
In particular, the measures $\mu_1$ and $\mu_2$ are finite. Combining~\eqref{Psi_3} with~\eqref{mu_3} we
conclude that
$$A:=\alpha_2-\int_\Real\frac{x}{1+x^2}\,d\mu_2(x)=-\alpha_1+\int_\Real\frac{x}{1+x^2}\,d\mu_1(x).$$

Now the fact that $\Phi_1+A\in\mathfrak P_0$ follows from~\eqref{Psi_j} and Remark~\ref{GorBaBauer_Remark}.

Let $z\in\UH$. Denote~$w:=\Phi_1(z)+A$. Then $\Phi_2(w-A)-w=\Phi_3(z)-z-\big(\Phi_1(z)+A-z\big)$. It follows
that
\begin{equation}\label{Phi2}
|\Phi_2(w)-w|\le \frac{C}{\Im z}\le\frac{C}{\Im w}
\end{equation}
for some $C>0$, which does not depend on~$z$. Since~$z\in\UH$ is arbitrary, inequality~\eqref{Phi2} holds for
all $w$ such that $w-A\in\Phi_1(\UH)$. Since $\Phi_1\in\mathfrak P_0$,  by Lemma~\ref{H_epsilon} there
exists~$b>0$ such that $\UH_b\subset\Phi_1(\UH)$. Hence \eqref{Phi2} holds for all $w\in\UH_b$, which can be
reformulated in the following way: inequality \eqref{Phi2} holds for all $w\in\UH$ and $\Phi_2$ replaced by the
function $\Phi_{2,b}(w):= \Phi_2(w+ib)-ib$. Hence $\Phi_{2,b}\in\mathfrak P_0$. But by Lemma~\ref{mlemma} the
latter statement implies that the function~$\Phi_2$ also belongs to the class~$\mathfrak P_0$. The proof is now
finished.
\end{proof}

\begin{corollary}\label{P_0_comp_cor}
Suppose that  $\varphi_j$, $j=1,2,3$, are univalent self-mappings of~$\UD$, and
$\varphi_3=\varphi_2\circ\varphi_1$. If any two of these three functions belong to~$\mathcal P_0$, then so does
the third one.
\end{corollary}
\begin{proof}
If $\varphi_1,\varphi_2\in\mathcal P_0$, then $\varphi_3$ also belongs to~$\mathcal P_0$ because of the
semigroup property of the class~$\mathfrak P_0$, see Remark~\ref{GorBaBauer_Remark}.

Suppose now that $\varphi_3\in\mathcal P_0$ and one of the functions $\varphi_1$, $\varphi_2$ belongs
to~$\mathcal P_0$. Then by Lemma~\ref{three functions}, we have~$\varphi_j\in\mathcal P$ for $j=1,2,3$. Now one
can apply Proposition~\ref{P_0_composition} to deduce that $\varphi_j\in\mathcal P_0$ for $j=1,2$, which
finishes the proof.
\end{proof}

Another tool in the proof of Theorem~\ref{GorBaTheorem} is \cite[Theorem~4.2]{Bauer}, which, in particular,
contains the following
\begin{proposition}\label{Bauer}
Let $(G_n)$ be a sequence from $\mathfrak P_0$ such that for any $n,~m\in\Natural$ either ${G_n(\UH)\subset
G_m(\UH)}$, or ${G_n(\UH)\supset G_m(\UH)}$. If $\ell(G_n)$ has a finite limit as $n\to+\infty$, then $G_n$
converges to a function $G\in\mathfrak P_0$ uniformly on $\UH_\varepsilon:=\{z:\Im z>\varepsilon\}$ for each
$\varepsilon>0$, with
\begin{gather}
\label{ell}%
\ell(G)=\lim_{n\to+\infty}\ell(G_n),\\
\label{image}%
G(\UH)=\bigcup\limits_{\varepsilon>0}\,\bigcup\limits_{k\in\Natural}~\mathrm{int}\!\left(\bigcap\limits_{n>k}
G_n(\UH_\varepsilon)\right).
\end{gather}
\end{proposition}
\begin{remark}\label{BauerR}
If $F,G\in\mathfrak P_0$ and $F(\UH)\subset G(\UD)$, then $F=G\circ W$, where $W=G^{-1}\circ F$ is a
conformal mapping of $\UH$ into itself. It is actually known (see, e.\,g.,\,\cite[Lemma~4.1]{Bauer}) that
$W\in\mathfrak P_0$. Therefore, according to Remark~\ref{Julia},
$W(\UH_\varepsilon)\subset\UH_{\varepsilon}$ for any $\varepsilon>0$. As a consequence, the inclusion
$F(\UH)\subset G(\UH)$, where $F,G\in\mathfrak P_0$, implies that $F(\UH_\varepsilon)\subset
G(\UH_\varepsilon)$, for any $\varepsilon>0$. It also follows that for $F,G\in\mathfrak P_0$ the equality
$F(\UH)=G(\UH)$ occurs only if~$F=G$.
\end{remark}

\subsection{Proof of Theorem~\ref{GorBaTheorem}}
First of all let us note that the implication~$(i)\Rightarrow(ii)$ is almost trivial. Indeed, assume (i) holds
and find for each~$t\ge0$ a number $\theta=\theta(t)\ge0$ such that $f_{\theta(t)}(\UD)=\Omega_t$. We can assume
that $\theta(0)=0$. Then functions $F_t:=f_{\theta(t)}$ form the family one have to construct in order to
prove~(ii).

Let us now assume that~(ii) holds. We are going to prove~(i). To obtain this implication we switch to the
upper half-plane~$\UH$ as a reference domain. So assume that for each $t\ge0$ the function~$F_t$ maps
$\UH$ conformally onto~$\Omega_t$ and
\begin{equation}\label{F_t_UH}
\big(F_t^{-1}\circ F_0\big)(z)=a(t)z+b(t)+\frac{c(t)}z+\gamma_t(z),\quad \angle\lim_{z\to\infty}z\gamma_t(z)=0,
\end{equation}
where $a(t)>0$ and $b(t)\in\Real$.

The proof follows the same scheme as for Theorem~\ref{par_th}. First of all we take M\"obius transformations
$L_t(z):=a(t)z+b(t)$ and define $H_t:=F_t\circ L_t$. Denote $\Phi_{s,t}:=H_t^{-1}\circ H_s$. We claim that
\begin{equation}
\label{Phi_st}\Phi_{s,t}\in\mathfrak P_0,\quad \text{for all }s\ge0,~t\ge s.
\end{equation}
Indeed, for $s=0$ the above assertion follows from~\eqref{F_t_UH} and the fact that $a(0)=1$, $b(0)=0$. By the
same reason $\Phi_{0,s}\in\mathfrak P_0$. Now~\eqref{Phi_st} follows from the equality
$\Phi_{s,t}\circ\Phi_{0,s}=\Phi_{0,t}$ and Corollary~\ref{P_0_comp_cor}.

Let us denote $x(t):=\ell(\Phi_{0,t})$, $t\ge0$. Since the functional $\ell$ is additive,
\begin{equation}\label{x(t)}
x(t)=x(s)+\ell(\Phi_{s,t})\quad\text{ for all $s\ge0$ and $t\ge s$.}
\end{equation}
By~\cite[Lemma~4.1]{Bauer}, it follows that $t\mapsto x(t)$ is non-decreasing and
\begin{equation}\label{x}
x(t)=x(s) \Longleftrightarrow \Phi_{s,t}=\id_\UH.
\end{equation}

Fixing the parameter $t$ and regarding the parameter~$s$ in equality~\eqref{x(t)} as a variable, one can
use Proposition~\ref{Bauer} along with Remarks~\ref{inclusion_chain_Remark} and~\ref{BauerR} to see that
the function $[0,+\infty)\ni t\mapsto x(t)\in[0,+\infty)$ is continuous.

Let us now set $h_t:=H_t\circ H$, $\phi_{s,t}:=H^{-1}\circ\Phi_{s,t}\circ H$, $t\ge0$, $s\in[0,t]$, where
$H$ stands for the Cayley map~\eqref{Cayley} sending of~$\UD$ onto~$\UH$ with $H(0)=i$ and $H(1)=\infty$.
The rest of the proof is almost the same as Steps~3 and~4 in the proof of Theorem~\ref{par_th}, except
that~$x(t)$ can be now unbounded and instead of Proposition~\ref{m_lemma} we apply
Proposition~\ref{GorBaLemma}. By this reason, we omit the details. \proofbox

\section{Sufficient conditions for chordal admissibility}
\label{Geometric properties} The condition of conformal embedding at a prime end, contained in the definition of
chordally admissible families of simply connected domains, involves in an essential way information on boundary
behaviour of conformal mappings related to these domains. This might make direct application of
Theorem~\ref{Main_Theorem} in many cases quite difficult. That is why it is interesting to establish some
sufficient conditions for chordal admissibility of purely geometric nature. An example of such a sufficient
condition is the following theorem.

A closed Jordan curve $C\subset\Complex$ in the complex plane~$\Complex$ is said to be {\it Dini-smooth} if there exists a bijective mapping~$h$ of $\UC:=\partial\UD$ onto~$C$ such that
the function $\Real\ni\theta\mapsto h_1(\theta):=h(e^{i\theta})\in C$ has non-vanishing Dini-continuous derivative. A closed Jordan curve $C$ in the Riemann sphere~$\ComplexE$ is called {\it Dini-smooth} if for some (or equivalently for all) M\"obius transformation $T$ sending $C$ into~$\C$, the image $T(C)$ of $C$ is a {\it Dini-smooth} curve.

\begin{theorem}\label{chordal_adm_criterium}
Let $(\Omega_t)$ be an L-admissible family and $p\in\partial_\infty\Omega_0$. Suppose that the following
conditions hold:
\begin{itemize}

\item[(i)] there exists a Dini-smooth closed Jordan curve $C$ in the Riemann sphere
such that $p\in C$ and one of the two connected components of~$\ComplexE\setminus C$ is contained
in~$\Omega_0$.

\item[(ii)] for each $t>0$ there exists a Dini-smooth closed Jordan curve $C_t$ in the Riemann sphere
such that $p\in C_t$ and $\Omega_t\cap C_t=\emptyset$.
\end{itemize}
Then the family $(\Omega_t)$ is chordally admissible.
\end{theorem}

\begin{proof}

Denote by $D$ the connected component of~$\ComplexE\setminus C$ that is contained
in~$\Omega_0$. Since $D$ is a Jordan domain, there exists a conformal mapping~$f_1$ of $\UD$ onto $D$ continuously extandable to~$\partial\UD$ with $f_1(1)=p$.  Let us denote by $\Gamma$ the curve $[0,1)\ni x\mapsto f_1(x)\in D\subset\Omega_0$. Obviously, we have $\Gamma(x)\to p$ as $x\to1-$.

We note that since $p\in\partial_\infty\Omega_0$,  condition $(ii)$ in the theorem implies that
$p\in\partial_\infty\Omega_t$ for all~$t\ge0$.

The remaining part of the proof will be divided into three steps.

\Step1{For each $t\ge0$ there exists a conformal mapping $F_t$ of $\D$ onto $\Omega_t$ such that the following
two angular limits exist in the Riemann sphere:
\begin{align}\label{anglim_1}
 &\angle\lim_{z\to 1}F_t(z)=p, \\
 &\angle\lim_{z\to 1}\Phi_t(z)=1,\quad \text{ where } \Phi_t:=F_t^{-1}\circ F_0. \label{anglim_2}
\end{align}
}

\noindent{\it Proof of Step 1}: Fix $t\ge0$ and let $G_t$ be any conformal mapping of~$\UD$ onto~$\Omega_t$.
Bearing in mind that $p\in\partial_\infty\Omega_t$ and that $G_t$ is univalent, we find in accordance with
Theorem~\ref{Koebe} that the curve $\gamma_t:[0,1)\to\UD$; $x\mapsto G_t^{-1}\big(\Gamma(x)\big)$ lands at some
particular point $\xi_t\in\UC$, i.e., $\gamma_t(x)\to\xi_t$ as $x\to 1-$. Note also that
$G_t(\gamma_t(x))=\Gamma(x)\to p$ as $x\to 1-$. Therefore, $p$ is an asymptotic value for $G_t$ at the point
$\xi_t$. Since $G_t$ is univalent, it is also a normal function (see \cite[Lemma~9.3]{Pommerenke}). Therefore,
by Lehto--Virtanen's Theorem (see, e.g., \cite[Theorem~9.3]{Pommerenke}) $G_t$ has angular limit at $\xi_t$
equal to $p$.

Now, define $F_t(z):=G_t(\xi_t z)$, $z\in\D$. Statement~\eqref{anglim_1} clearly holds. To deal with
statement~\eqref{anglim_2}, we fix $t>0$. Likewise, we denote $\gamma(x):=F_0^{-1}\big(\Gamma(x)\big)$.
Appealing again to Theorem~\ref{Koebe}, we conclude that $\gamma$ is a Jordan arc with $\gamma(x)\to 1$ as $x\to
1-$. Now, note that $\Phi_t(\gamma(x))=\overline{\xi_t}\gamma_t(x)$. Hence, $\Phi_t$ has an asymptotic value
at~$z=1$, equal to~$1$. Once again, by \cite[Lemma~9.3]{Pommerenke} and Lehto--Virtanen's Theorem, we conclude
that the angular limit of $\Phi_t$ at the point~$z=1$ equals~$1$.

\Step2{For each $t>0$ the function $\Phi_t$ has finite angular derivative at the point~$z=1$.}

Observe that $\Phi_t$ is a holomorphic self-map of the unit disk having
an (angular) fixed point at~$z=1$ by~\eqref{anglim_2}. According to Remark~\ref{ex_of_ad}, the following limit
exists
$$
\Phi'_t(1):=\angle\lim_{z\to1}\frac{\Phi_t(z)-1}{z-1}\in(0,+\infty)\cup\{\infty\}.
$$
Further we proceed by contradiction. So we assume that for some $t>0$ we have ${\Phi'_t(1)=\infty}$. Since $\Phi_t$ is univalent, one can again take advantage of~\eqref{anglim_2} and~\cite[Theorem~10.5 on p.\,305]{Pommerenke} to conclude that $\angle\lim_{z\to1}\Phi'_t(z)=\infty$.

Since $C_t\cap\Omega_t=\emptyset$, the domain $\Omega_t$ is contained in one of the connected components of~$\ComplexE\setminus C_t$, which we denote by $D_t$. Applying if necessary the M\"obius  transformation $T(z):=1/(z-z_0)$, where $z_0\in\Complex$ is any point in the exterior of $D_t$, we can assume that $D_t\subset\Complex$ and is a bounded domain. Then $D$ is also bounded, with $\partial D=C$ being a Jordan curve in the plane. Note that applying $T$ will change $f_1$, $F_0$ and $F_t$, but not the function~$\Phi_t$.

Since by condition (i), $C$ is a Dini-smooth Jordan curve,  from \cite[Theorem~3.5]{Pommerenke-II} it follows
that there exists finite angular derivative~$f_1'(1)\neq0$. At the same time,
$$
f_1(\D)=D\subset \Omega_0=F_0(\D).
$$
Recall that $F_0(1)=p$ in the angular sense. Applying \cite[Theorem~4.14]{Pommerenke-II}, we conclude that $F_0$
has finite angular derivative at~$z=1$. Since $F_0$ is univalent, we have (see again
\cite[Theorem~10.5]{Pommerenke})
$$
F'_0(1)=\lim_{r\to1-}F'_0(r)\in \C.
$$
Differentiating the identity $F_t\circ\Phi_t(r)=F_0(r)$, $r\in(0,1)$,  and passing to limits as $r\to1-$, we
deduce that $\lim_{r\to1-}F'_t(\Phi_t(r))=0$. Now, by~\eqref{anglim_2} we see that $r\in[0,1)\mapsto
\Phi_t(r)\in\D$ is a Jordan arc in $\D$ landing at $z=1$. Therefore, $0$ is an asymptotic value of $F'_{t}$ at
the point~$z=1$. Since $F_t$ is univalent, its derivative $F'_{t}$ is a normal function (see, e.g.,
\cite[Lemma~9.3]{Pommerenke}). Therefore, from Lehto--Virtanen's Theorem (see, e.g.,
\cite[Theorem~9.3]{Pommerenke}) it follows that  $\angle\lim_{z\to1}F'_t(z)=0$.

Now let us consider the domain~$D_t$. We argue in a similar way as above. Namely, we let $f_2$ be any conformal
mapping of $\UD$ onto $D_t$. The  $C_t=\partial D_t$ is a Dini-smooth Jordan curve in~$\Complex$ and $p\in C_t$.
Therefore, from \cite[Theorem~3.5]{Pommerenke-II}, it follows that $f_2(\zeta_2)=p$ for some (in fact the unique
one) $\zeta_2\in\UC$ and that there exists finite angular derivative~$f'_2(\zeta_2)\neq0$. Recall that
$F_{t}(1)=p$ in the angular sense. Therefore, applying \cite[Theorem~4.14]{Pommerenke-II}, we conclude that
$|f'_2(\zeta_2)|=c|F'_t(1)|$ for some $c>0$. However, the left-hand side of this equality is different from~$0$
while~the right-hand side vanishes. This contradiction proves Step~2.

\Step3{The family $(\Omega_t)$ is chordally admissible.} Just apply Lemma \ref{adm} using previous Step 2.
\end{proof}

Now let us formulate an analog of Theorem~\ref{chordal_adm_criterium} for Goryainov--Ba evolution families. Let
us introduce some definitions.

Let $n\in\Natural$ and $\alpha\in(0,1)$. A closed Jordan curve $C$ in the complex plane~$\Complex$ will be called {\it $C^{n,+0}$-smooth} if there is a bijective mapping~$h$ of $\UC:=\partial\UD$ onto~$C$ such that
the function $\Real\ni\theta\mapsto h_1(\theta):=h(e^{i\theta})\in C$ has derivatives up to the order $n$,
the first derivative~$h_1'$ does not vanish, and $h_1^{(n)}$ is H\"older continuous with some
exponent~$\alpha>0$. A closed Jordan curve $C$ in the Riemann sphere~$\ComplexE$ is called {\it
$C^{n,+0}$-smooth} if for some (or equivalently for all) M\"obius transformation $T$ sending $C$
into~$\C$, the image $T(C)$ of $C$ is a $C^{n,+0}$-smooth curve.

Two $C^1$-smooth Jordan curves $C_1,C_2\in\Complex$ with a common point~$p$ is said to {\it have contact
of order~$n\in\Natural$ at the point~$p$},  if there exist $C^1$-smooth
parameterizations~${w_1:[-1,1]\to\Complex}$, ${w_2\in[-1,1]\to\Complex}$ of the curves~$C_1$ and $C_2$,
respectively, such that ${w_1(0)=w_2(0)=p}$, ${w_1'(0),w_2'(0)\neq0}$, and ${|w_1(t)-w_2(t)|=o(t^n)}$
as~$t\to0$. Again using M\"obius transformations one can extend this definition to the case of Jordan
curves in the Riemann sphere and $p=\infty$.

\begin{theorem}\label{GB_adm_criterium}
Let $(\Omega_t)$ be an L-admissible family and $p\in\partial_\infty\Omega_0$. Suppose that the following
conditions hold:
\begin{itemize}

\item[(i)] there exists a $C^{3,+0}$-smooth closed Jordan curve $C$ in the Riemann sphere
such that $p\in C$ and one of the two connected components of~$\ComplexE\setminus C$ is contained
in~$\Omega_0$.

\item[(ii)] for each $t>0$ there exists a $C^{3,+0}$-smooth closed Jordan curve $C_t$ in the Riemann sphere
such that
\begin{itemize}
\item[(ii.1)] $p\in C_t$,
\item[(ii.2)] the curves $C$ and $C_t$ have second order contact at the point~$p$, and
\item[(ii.3)]  $\Omega_t\cap C_t=\emptyset$.
\end{itemize}
\end{itemize}
Then the family $(\Omega_t)$ satisfies condition~(ii) in Theorem~\ref{GorBaTheorem}.
\end{theorem}

For the proof of the above theorem we need the following elementary lemmas.

\begin{lemma}\label{2ndorder1}
Let $\varphi:\UD\to\UD$ be a holomorphic univalent function having $C^3$-smooth injective extension
to~$\partial\UD$ with $\varphi(1)=1$. If the curve $\partial\varphi(\UD)$ has
second order contact with~$\partial\UD$ at the point~$z=1$, then there exists $h\in\Moeb(\UD)$ such that the
function~$h\circ\varphi$ belongs to the class~$\mathcal P_0$.
\end{lemma}
\begin{proof}
This statement is easier to prove if we pass to the upper half-plane. So we consider the function
$\Phi:=H\circ\varphi\circ H^{-1}$, where as earlier $H$ stands for the Cayley map~\eqref{Cayley}. Using
the facts that $\varphi$ has $C^3$-smooth extension to~$\partial\UD$ and that $\varphi(1)=1$, we can
assume that the following expansion takes place:
\begin{equation}\label{Phi_expan}
\Phi(z)=az+b+\frac{c}{z}+\gamma(z),~~~z\in\UH\cup\Real,
\quad a\neq0,~~\lim_{\substack{z\to\infty,\\[.1ex]z\in\UH\cup\Real}}z\gamma(z)=0.
\end{equation}

The fact that $\partial\varphi(\UD)$ has second order contact with~$\partial\UD$ at~$z=1$ can be
reformulated in the following way:
\begin{equation}\label{curve_geom}
\Im \Phi(t)\to 0\quad\text{as}\quad\text{$t\to\infty$, $t\in\Real$.}
\end{equation} (The proof of this claim is elementary and so we omit it.)

Combining~\eqref{Phi_expan} with \eqref{curve_geom} and taking into account that~$\Phi(\UH)\subset\UH$, we conclude that~$a>0$ and $b\in\Real$. Denote $L(z):=(z-b)/a$.
It follows that~$L\circ\Phi\in\mathfrak P_0$. Consequently, $h\circ\varphi\in\mathcal P_0$, where
$h:=H^{-1}\circ L\circ H$, which completes the proof.
\end{proof}

\begin{lemma}\label{2ndorder2}
Let $D_1$ and $D_2$ be domains in the Riemann sphere bounded by $C^{3,+0}$-smooth Jordan curves $C_1$ and
$C_2$ respectively. Suppose that $D_1\subset D_2$ and that the intersection $C_1\cap C_2$ contains a
point~$p$ at which the curves~$C_1$ and $C_2$ have contact of second order. If $f_1$ maps
conformally~$\UD$ onto~$D_1$ in such a manner that $f_1(1)=p$, then there exists a conformal mapping~$f_2$
of~$\UD$ onto~$D_2$ such that the function~$f_2^{-1}\circ f_1$ belongs to the class~$\mathcal P_0$.
\end{lemma}
\begin{proof}
First of all we note that because of the smoothness of the curves~$C_1$ and $C_2$, any conformal mappings $f_1$
and $f_2$ of~$\UD$ onto~$D_1$ and $D_2$ respectively, have $C^{3,\alpha}$-smooth extension to the unit circle
with some~$\alpha>0$, see, e.g.,~\cite[Theorem~3.6]{Pommerenke-II}. Assume now, according to the conditions of
the lemma, that $f_1(1)=p$. We also can choose~$f_2$ in such a way that $f_2(1)=p$. Then the function
$\varphi:=f_2^{-1}\circ f_1$ has $C^3$-smooth extension to~$\partial\UD$ and maps $\UD$ onto a subdomain
of~$\UD$ bounded by a curve having second order contact point with~$\partial\UD$ at~$z=1$. Then by
Lemma~\ref{2ndorder1}, $h\circ\varphi\in\mathcal P_0$ for some $h\in\Moeb(\UD)$. Hence, we may finish the proof
by replacing $f_2$ with $f_2\circ h^{-1}$.
\end{proof}

\begin{lemma}\label{simple_lemma}
Let $U\subset\UD$ be a simply connected domain and
$\varphi\in\mathcal P$. If $\varphi(\UD)\subset U$, then there
exists $\varphi_2\in\mathcal P$ such that $\varphi_2(D)=U$.
\end{lemma}
\begin{proof}
Let $f_2$ be any conformal mapping of $\UD$ onto~$U$ and $\Gamma:=\varphi\big([0,1)\big)$.  Then the curve
$\Gamma$ lands at the point $z=1$, which is a common boundary point for the domains~$\varphi(\UD)$ and~$U$.
According to Theorem~\ref{Koebe} the curve~$\Gamma_2:=f_2^{-1}(\Gamma)$ lands at some particular point
$\zeta_2\in\partial\UD$. Now we can apply \cite[Theorem~4.14]{Pommerenke-II} for $f_1:=\varphi$ to deduce that
there exists finite angular derivative $f_2'(\zeta_2)$ and therefore there exists the angular limit
$\angle\lim_{\zeta\to\zeta_2}f_2(\zeta)$, which equals~$1$ because by just mentioned Theorem~\ref{Koebe} the
image of $\Gamma':=f_2\big([0,1)\big)$ can not land on $\partial f_2(\UD)$ at a point different from the landing
point of $\Gamma=f_2(\Gamma_2)$.

Let us consider the function $g(z):=f_2(z/\zeta_2)$. There exist
$g(1)=1$ and $g'(1)\neq\infty$ in angular sense. Moreover,
according to Remark~\ref{ex_of_ad}, $g'(1)>0$. To complete the
proof it suffices now to set $\varphi_2:=g\circ\ell_{x}$, where
$$\ell_x(z):=\frac{z-x}{1-xz},\quad x:=\frac{1-g'(1)}{1+g'(1)}.$$
\end{proof}

\begin{proof}[Proof of Theorem~\ref{GB_adm_criterium}]
Denote by $D$ the connected component of~$\ComplexE\setminus C$
that is contained in~$\Omega_0$. Since $D$ is a Jordan domain,
there exists a conformal mapping~$F$ of $\UD$ onto $D$
continuously extendable to~$\partial\UD$ with $F(1)=p$.

Fix any $t>0$. Since $C_t$ does not intersect $\Omega_t$, one of
the connected components of~$\ComplexE\setminus C_t$ contains
$\Omega_t$. Denote this connected component by $D_t$ and apply
Lemma~\ref{2ndorder2} with $F$, $D$, and $D_t$ substituted
for~$f_1$, $D_1$, and $D_2$, respectively. By this lemma there
exists a conformal mapping $f_2$ of $\UD$ onto $D_t$ such that
$\varphi:=f_2^{-1}\circ F$ belongs to $\mathcal P_0$. By
construction ${D\subset\Omega_0\subset\Omega_t\subset f_2(\UD)}$.
Let $U:=f_2^{-1}(\Omega_t)$ and $U_0:=f_2^{-1}(\Omega_0)$. Note
that ${\varphi(\UD)\subset U_0\subset U}$. Then by
Lemma~\ref{simple_lemma} there exist
$\varphi_2,\varphi_2^0\in\mathcal P$ mapping $\UD$ conformally
onto $U$ and $U_0$, respectively. Now using Lemma~\ref{three
functions} we conclude that $\varphi=\varphi^0_2\circ\varphi^0_1$
and $\varphi^0_2=\varphi_2\circ\varphi_1$ for some functions
$\varphi_1,\varphi_1^0\in\mathcal P$.  By
Proposition~\ref{P_0_composition} with $\varphi$, $\varphi_2^0$,
and $\varphi_1^0$ substituted for $\varphi_3$, $\varphi_2$, and
$\varphi_1$, respectively, there exists
$h_0\in\Moeb(\UD)\cap\mathcal P$ such that the functions
$h_0\circ\varphi_1^0$ and $\varphi_2^0\circ h_0^{-1}$ belong to
the class~$\mathcal P_0$. Finally, we apply
Proposition~\ref{P_0_composition} with $\varphi_2^0\circ
h_0^{-1}$, $\varphi_2$, and $\varphi_1\circ h_0^{-1}$ substituted
for $\varphi_3$, $\varphi_2$, and $\varphi_1$, respectively, to
conclude that there exists $h\in\Moeb(\UD)\cap\mathcal P$ such
that $h\circ \varphi_1\circ h_0^{-1}\in\mathcal P_0$.

Define $F_t^0:=f_2\circ\varphi_2^0\circ h_0^{-1}$ and
$F_t:=f_2\circ\varphi_2\circ h^{-1}$.  These functions are
conformal mappings of~$\UD$ onto domains $\Omega_0$ and
$\Omega_t$, respectively, with
\begin{equation}\label{inP0}
F_t^{-1}\circ F_t^0=h\circ \varphi_1\circ h_0^{-1}\in\mathcal P_0.
\end{equation}
One more important relation follows from the above construction.
Namely,  denote ${\varphi_t:=h_0\circ \varphi_1^0}$. Then
\begin{equation}\label{FF}
F_t^0\circ \varphi_t=F.
\end{equation}

We have constructed a family $(F_t^0)$ of conformal mappings
of~$\UD$  onto $\Omega_0$. Let us prove that~$F_t^0$ does not
depend on~$t$. To this end take arbitrary $s,t>0$ and write
${h_{s,t}:=\big(F^0_t\big)^{-1}\circ F_s^0}$. This function
maps~$\UD$ conformally onto itself, i.e., $h_t\in\Moeb(\UD)$. Now
recall that~$F$ does not depend on $t$. According to~\eqref{FF},
it follows that $\varphi_t=h_{s,t}\circ \varphi_s$. Recall also
that~$\varphi_t\in\mathcal P_0$. Therefore, by
Corollary~\ref{P_0_comp_cor} with $\varphi_t$, $h_{s,t}$, and
$\varphi_s$ substituted for $\varphi_3$, $\varphi_2$, and
$\varphi_1$, respectively, we have $h_{s,t}\in\mathcal P_0$. Since
$\Moeb(\UD)\cap \mathcal P_0=\{\id_\UD\}$, we conclude that
$h_{s,t}=\id_\UD$ and therefore $F_s^0=F_t^0$. Now in view
of~\eqref{inP0} we may finish the proof just by
setting~$F_0:=F_t^0$.
\end{proof}

The cases when Theorems~\ref{chordal_adm_criterium} and~\ref{GB_adm_criterium} can  be applied are quite
different from the classical setting of slit mappings. The following corollary of Theorem~\ref{GorBaTheorem} is
a simple assertion covering, in an obvious way, the case of an L-admissible family obtained by erasing a slit in
the complex plane or in the half-plane. More general case of a slit in an arbitrary simply connected domain can
be treated by using a conformal mapping onto the half-plane.

In the formulation and the proof we use the theory of primes ends.
The main definitions and some basic facts from this theory can be
found, e.g., in~\cite[Chapter~9]{ClusterSets}
or~\cite[Chapter~2]{Pommerenke-II}.
\begin{corollary}\label{GorBaCorSimple}
Let $(\Omega_t)$ be an L-admissible family and $P_0$ any prime end
of the domain~$\Omega_0$. Suppose that $P_0$ is degenerate, i.e.,
the impression of~$P_0$ consists of one point~$p_0$, and for each
$t\ge0$ there exists $\varepsilon>0$ such that
\begin{equation}\label{epsilon}\partial\Omega_t\cap
D(p_0,\varepsilon)=\partial\Omega_0\cap
D(p_0,\varepsilon),\end{equation} where $D(p_0,\varepsilon)$
stands for the disk of radius~$\varepsilon$ centered at~$p_0$.
Then assertions~(i) and (ii) in Theorem~\ref{GorBaTheorem} hold
and, in particular, the family~$(\Omega_t)$ is chordally
admissible.
\end{corollary}

\begin{proof}
Let $F_0$ be a conformal mapping of~$\UD$ onto~$\Omega_0$. The
fact that the impression $I(P_0)$ of the prime end~$P_0$ consists
of only one point~$p_0$, implies that $F_0$ has unrestricted limit
equal to~$p_0$ at some point $\zeta_0\in\UC$ and the curve
$F_0\big(\zeta_0[0,1)\big)$ tends to $P_0$. Using rotation
of~$\UD$ we can assume that~$\zeta_0=1$.

Now let us fix $t>0$ and let $F_t$ be a conformal mapping of~$\UD$
onto~$\Omega_t$. Another consequence of the degeneracy of~$P_0$ is
that for any $\varepsilon>0$ there exists a
null-chain~$(C_n)_{n\in\Natural}$ belonging to the prime end~$P_0$
such that for any $n\in\Natural$ the connected component~$V_n$ of
the set~$\Omega_0\setminus C_n$ containing~$C_{n+1}$, lies in
$D(p_0,\varepsilon)$ along with its boundary. Indeed, take any
null-chain~$(C_n)_{n=0}^{+\infty}$ belonging to the prime
end~$P_0$. Since by definition of the impression of prime end,
$$\{p_0\}=I(P_0)=\bigcap_{n\in\Natural}\overline{V_n},$$ we have
$V_n\subset D(p_0,\varepsilon)$ for all~$n$ sufficiently large. Therefore, dropping a finite number of
cross-cuts $C_n$ we get the desired null-chain with all~$\overline{V_n}$ lying in~$D(p_0,\varepsilon)$.

From~\eqref{epsilon} it follows that $$\Omega_t\cap D(p_0,\varepsilon)=\Omega_0\cap D(p_0,\varepsilon).$$ The
latter in its turn implies that $(C_n)$ is a null-chain for the the domain~$\Omega_t$ as well and that for each
$n\in\Natural$ the connected component of~$\Omega_t\setminus C_n$ that contains $C_{n+1}$ coincides with~$V_n$.
It follows that the prime end of~$\Omega_t$ defined by the null-chain~$(C_n)$ is also degenerate and its
impression is~$\{p_0\}$. Hence there exists a point~$\zeta_t\in\UC$ such that the unrestricted limit
$\lim_{z\to\zeta_t} F_t=p_0$ and the curve $F_t\big(\zeta_t[0,1)\big)$ tends to $P_t$. Again using rotation
of~$\UD$ we can assume that~$\zeta_t=1$.

Given a conformal mapping $F$ of the unit disk~$\UD$, the preimage
of any null-chain for the domain~$F(\UD)$ is a null-chain
for~$\UD$. Therefore, for each $n\in\Natural$ the set
$U_{0,n}:=F_0^{-1}(V_n)$ is a Jordan domain such that
$\Gamma_{0,n}:=\UC\cap\partial U_{0,n}$ is an arc and contains
$\zeta=1$ as an internal point. By the same reason
$U_{t,n}:=F_t^{-1}(V_n)$ is also a Jordan domain such that
$\Gamma_{t,n}:=\UC\cap\partial U_{t,n}$ is an arc, which contains
$\zeta=1$ as an internal point. This means that the function
$\Phi_t:=F_t^{-1}\circ F_0$ can be continuously extended to all
internal points of the arc~$\Gamma_{0,1}$ and that
$\Phi_t(\Gamma_{0,n})=\Gamma_{t,n}$ for all $n\in\Natural$, $n>1$.
In particular, it follows that $\Phi_t$ extends analytically to
the point~$z=1$. Since all prime ends of~$\UD$ are degenerate,
$$\bigcap_{n\in\Natural} \Gamma_{0,n}=\{1\}=\bigcap_{n\in\Natural}
\Gamma_{t,n}.$$ Hence we have that~$\Phi_t(1)=1$.

Summarizing the above facts one can easily conclude that $\Phi_t\in\classCC$. Therefore,~assertion~(ii) in
Theorem~\ref{GorBaTheorem} holds, which implies, according to this theorem, assertion~(i). The chordal
admissibility follows from assertion~(ii) by Lemma~\ref{adm}. The proof is finished.
\end{proof}

\section{An example}
\label{the example}

Let us recall that a holomorphic function $\varphi :\D\rightarrow \C $ belongs to the disk algebra~$\mathcal A$
if it has a continuous extension to the closed unit disk. In this section we construct an example of an
evolution family~$(\varphi_{s,t})$ which is contained in the disk algebra, but has no associated Loewner chains
$(f_t)$ with locally connected boundaries of the image domains~$f_t(\UD)$. More precisely, $\partial f_t(\UD)$
is not locally connected for every~$t\ge0$ and each Loewner chain~$(f_t)$ associated with~$(\varphi_{s,t})$.
Another interesting property of the evolution family we construct below is as follows: $(\varphi_{s,t})$ is a
Goryainov--Ba (and in particular, chordal) evolution family and $\partial f_t(\UD)$ is not locally connected,
because $f_t$ fails to have an angular limit at the point~$z=1$.

Before we start constructing the announced example let us introduce some notations and results we are going to
use.

As earlier we denote by $\Prend(D)$ the set of all prime ends of a simply connected domain~$D$. Let $U$
and $W$ be simply connected domains and $F$ a conformal mapping of~$U$ onto~$W$. By $F^{p.e.}$ we will
denote the bijective map between $\Prend(U)$ and $\Prend(W)$ induced by~$F$. If $U$ is $\UH$ or $\UD$, we
will identify $\Prend(U)$ with $\partial_\infty U$.

As in Section~\ref{Goryainov-Ba} we will denote by $\mathcal H$
the class of all univalent holomorphic self-mappings~$G$ of~$\UH$
such that
$$\lim_{\substack{z\to\infty\\z\in\UH}}\big(G(z)-z\big)=0.$$
\begin{remark}\label{H-remark}
The Riemann Mapping Theorem and the Prime End Theorem (see e.g. \cite[p. 18]{Pommerenke-II}) imply:
\begin{itemize}
\item[(i)] For any compact set $E\subset
\UH\cup\Real$  there exists $G\in\mathcal H$ such that~$G(\UH)=\UH\setminus E$.
\item[(ii)] If $G,H\in\mathcal H$ and $G(\UH)=H(\UH)$, then $G=H$.
\item[(iii)] In particular, by assertion~(ii), the function~$G$ from assertion~(i) is unique. We will denote it
by~$G[E]$. This function can be continued by symmetry to the lower half-plane and consequently can be
represented in a neighbourhood of $\infty$ by a Laurent series with real coefficients and principal part of the
form~$az+b$, $a>0$.
\end{itemize}
\end{remark}

\begin{definition}
A Jordan arc $\Gamma$ with two different end-points~$a_1$ and
$a_2$ is said to be a {\it slit} in a domain~$D$ if
$\Gamma\subset\overline D:=D\cup\partial D$ and
$\Gamma\cap\partial\D=\{a\}$, $a=a_j$ for $j=1$ or $j=2$. The
point $a$ is called the {\it root} of the slit~$\Gamma$ or its
{\it landing point}. By parametrization of the slit~$\Gamma$ we
will mean a homeomorphic mapping~$\gamma$ of a
segment~$[\alpha,\beta]$ onto $\Gamma$
with~$\gamma(\beta)\in\partial D$.
\end{definition}

The following result was proved in~\cite{Kufarev_etal}. Closely
connected results can be found in~\cite{Popova1954} and
\cite[Chapter IV\S 7, Theorem 2]{Aleksandrov}.

\begin{result} \label{ChordalRepresentation_in_Examples}
For any slit $\Gamma$ in the upper half-plane~$\UH$ landing at a
finite point, there exists a
parametrization~$\gamma:[0,T]\to\Complex$ of $\Gamma$  such that
the functions $\Phi_{0,t}:=G_t^{-1}\circ G_0$, where
$G_s:=G\big[\gamma\big([s,T]\big)\big]$, $s\in[0,T]$, have the
following Laurent expansion in a neighbourhood of~$z=\infty$,
\begin{equation}\label{expan}\Phi_{0,t}(z)=z-\frac{t}{z}+\ldots\end{equation}
Moreover, there exists a continuous function
$\lambda:[0,T]\to\Real$ such that for each $z\in\UH$ the function
$w(t):=\Phi_{0,t}(z)$, $t\in[0,T]$, is a solution to the chordal
Loewner equation
\begin{equation}\label{ChLE_in_Example}\frac{dw(t)}{dt}=\frac{1}{\lambda(t)-w(t)},\quad t\ge0.\end{equation}
\end{result}

Now let us consider the domain $\Delta_0$ obtained by removing from the unit disk~$\UD$ the spiral curve
$C:[0,+\infty)\to\UD$ given by $C(\tau)= e^{i\tau}\big(1-1/(\tau+2)\big)$, tending asymptotically to the unit
circle. Denote $\Delta_\tau:=\UD\setminus C\big([\tau,+\infty)\big)$. For each $\tau\in[0,+\infty)$ all the
prime ends of the domain~$\Delta_\tau$ are trivial (i.e., have impressions consisting of a unique point) except
for exactly one prime end, impression of which coincides with the unit circle~$\partial\UD$. Denote this prime
end by $P_\tau$. Let us denote by $F_0$ the conformal mapping of $\UH$ onto $\Delta_0$ normalized by the
conditions $F_0(i)=0$, $F^{p.e.}_0(\infty)=P_0$.

We claim that for each $\tau\ge0$ there exists a conformal mapping
$F_\tau$ from~$\UH$ onto~$\Delta_\tau$ such that
$\tilde\Phi_{0,\tau}:=F_\tau^{-1}\circ F_0$ belongs to~$\mathcal
H$. To see this, let us first consider the conformal mapping
$H_\tau$ of~$\UH$ onto~$\Delta_\tau$ normalized by the
conditions~$H_\tau(i)=0$, $H^{p.e.}_\tau(\infty)=P_\tau$. Further,
the inverse mapping~$H^{-1}_\tau$ extends continuously to the
point~$C(\tau)$. Hence the function $\Psi_\tau:=H_\tau^{-1}\circ
F_0$ maps~$\UH$ onto~$\UH$ with slit along the Jordan
curve~$H_\tau^{-1}\big(C([0,\tau])\big)$ landing at the point
on~$H_\tau^{-1}\big(C(\tau)\big)\in\Real$. Note also that
$\Psi_\tau(\infty)=\infty$. By the Continuity theorem (see,
e.\,g., \cite[p.\,18]{Pommerenke-II}), the map $\Psi_\tau$ is
continuous up to the boundary. The argument of
Remark~\ref{H-remark} shows that in a neighbourhood of~$\infty$
the function $\Psi_\tau$ has a Laurent expansion of the form
$$\Psi_\tau(z)=az+b+\frac{c_1}{z}+\ldots,$$ where $a>0$ and $b\in\Real$. It follows that
$F_\tau(z):=H_\tau(az+b)$ is the desired conformal mapping.

Functions~$\tilde\Phi_{0,\tau}$ have a Laurent expansion of the
form~\begin{equation}\label{tilde}\tilde\Phi_{0,\tau}=z-\frac{t(\tau)}{z}+\ldots,\end{equation}
where $t(\tau)>0$.

Now let us fix some $\sigma>0$. Using the connection between $F_{\sigma}$ and $H_\sigma$ we conclude that the
function~$F_{\sigma}^{-1}\circ F_0$ maps $\UH$ onto $\UH\setminus\Gamma$, where
$\Gamma:=F_{\sigma}^{-1}\big(C([0,\sigma])\big)$ is a slit in~$\UH$ landing
at~$F_{\sigma}^{-1}(C(\sigma))\in\Real$. According to Theorem~\ref{ChordalRepresentation_in_Examples} there
exists a parametrization $t\mapsto\gamma(t)$ of the slit $\Gamma$ such that functions~$\Phi_{0,t}$ defined in
the statement of this theorem satisfy~\eqref{expan}. In particular, it means that for any $\tau\in[0,\sigma]$
there exists a unique $t_\tau\in[0,T]$ such that
\begin{equation}\label{reparam}
F_\sigma^{-1}\big(C(\tau)\big)=\gamma(t_\tau),\quad
\tau\in[0,\sigma],
\end{equation}
and the mapping~$[0,\sigma]\ni\tau\mapsto t_\tau\in[0,T]$ is an
increasing homeomorphism. Equality~\eqref{reparam} implies that
$\tilde\Phi_{0,\tau}=\Phi_{0,t_\tau}$. Consequently,
by~\eqref{expan}~and~\eqref{tilde}, $t(\tau)=t_\tau$ for
any~$\tau\in[0,\sigma]$. Since $\sigma>0$  can be chosen
arbitrarily, this means that $t\mapsto t(\tau)$ is a continuous
strictly increasing mapping of~$[0,+\infty)$ onto an interval of
the form~$[0,T_0)$, $T_0\in(0,+\infty]$.

Let us prove that $T_0=+\infty$.  Suppose on the contrary that $T_0<+\infty$.  Then it follows that $F_{\tau}$
converges, as $\tau\to+\infty$, uniformly on each compact subsets of~$\mathbb H$. Indeed, denote
$\tilde\Phi_{\nu,\mu}:=F_{\mu}^{-1}\circ F_{\nu}$, $\mu\ge\nu\ge0$. Since the functions $F_\tau$, $\tau>0$, form
a normal family in $\mathbb{H}$ and $\tilde\Phi_{\nu,\mu}(z)=z-\big(t(\mu)-t(\nu)\big)/z+\ldots$ in a
neighbourhood of $z=\infty$, for any compact set $K\subset\mathbb{H}$ we have
\begin{multline*}
\big|F_{\nu}(z)-F_{\mu}(z)\big|=\big|F_{\mu}(\tilde\Phi_{\nu,\mu}(z))-F_{\mu}(z)\big|\le\\
C^o_K\,|\tilde\Phi_{\nu,\mu}(z)-z|\le C_K (t(\mu)-t(\nu)),\quad z\in K,
\end{multline*}
provided $\mu>\nu$ and $\nu$ is large enough, say $\nu>\nu_K$,
where in the last inequality  we have applied that
$\tilde\Phi_{\nu,\mu}$ belongs to $\mathfrak P_0$. Here $C_K^o$,
$C_K$, and $\nu_K$ are positive constants independent of $\mu$ and
$\nu$. This proves convergence of~$F_\tau$ as $\tau\to+\infty$.

Denote the limit of~$F_\tau$ by $F_\infty$. For each fixed
$\nu\ge0$ there exists also a limit~$\tilde\Phi_{\nu,\infty}$
of~$\tilde\Phi_{\nu,\mu}$ as $\mu\to+\infty$. Indeed,  let
$\mu_2\ge\mu_1\ge\nu$ and $w:=\tilde\Phi_{\nu,\mu_1}(z)$. Then
$$\big|\tilde\Phi_{\nu,\mu_2}(z)-\tilde\Phi_{\nu,\mu_1}(z)\big|=\big|\tilde\Phi_{\mu_1,\mu_2}(w)-w\big|\le
\frac{\big(t(\mu_2)-t(\mu_1)\big)}{\Im w}\le
\frac{\big(t(\mu_2)-t(\mu_1)\big)}{\Im z}.$$ Moreover, by
\cite[Theorem~4.2]{Bauer}, $\tilde\Phi_{\nu,\infty}\in\mathfrak
P_0$. Now we will prove that $F_\infty$ maps $\mathbb H$ onto the
unit disk $\UD$ in one-to-one manner. First of all we recall that
$\cup_{\tau\ge0}\Delta_\tau=\UD.$ Hence, for any point $w_0\in
\UD$ there exists $\nu\ge0$ such that $F_{\nu}(z_{\nu})=w_0$ for
some $z_\nu\in\mathbb H$. Using the equality
$F_\nu(z)=F_\mu(\tilde\Phi_{\nu,\mu}(z))$, $z\in\mathbb H$, we
conclude that $z_\mu:=\tilde\Phi_{\nu,\mu}(z_\nu)$ solves the
equation~$F_\mu(z)=w_0$. Since $\tilde\Phi_{\nu,\mu}$ tends as
$\mu\to+\infty$ to a non-constant function, the point $z_\mu$
tends to an internal point $z_0\in\mathbb H$. Taking into account
locally uniform convergence of $F_\mu$ to $F_\infty$, we conclude
that $F_\infty(z_0)=w_0$. This applies to any $w_0\in \UD$. Hence,
$F_\infty(\mathbb H)\supset \UD$. On the other hand, the mapping
$F_\infty$ is univalent and maps $\mathbb{H}$ into $\UD$, because
it is a locally uniform limit of univalent functions mapping
$\mathbb{H}$ into $\UD$ and is not a constant function. This
proves that $F_\infty$ is a univalent mapping of $\mathbb H$ onto
$\UD$. Now we can pass to limits  in the equality
$F_0=F_\tau\circ\tilde\Phi_{0,\tau}$ in order to see that since
$\tilde\Phi_{0,\infty}(z)$ tends to $\infty$ as $z$ tends
to~$\infty$ along the positive direction of the imaginary axis,
the prime end $F^{p.e.}_0(\infty)$ contains a reachable point.
This conclusion contradicts the construction, and hence, proves
that $T_0=+\infty$.

Now let us define a family $(f_t:\UD\to\Complex)_{t\ge0}$ by the following relation
$$f_{t(\tau)}=F_\tau\circ H,\quad H(\zeta):=i\frac{1+\zeta}{1-\zeta},\quad \tau\ge 0.$$ Let
$\varphi_{s,t}:=f_t^{-1}\circ f_s$. Note that $\Phi_{0,t}=H\circ\varphi_{0,t}\circ H^{-1}$ for all
$t\in[0,+\infty)$. It follows that $\ell(\varphi_{s,t})=t-s$ for all $0\le s\le t <+\infty$. The family
$(\varphi_{s,t})$ satisfies~EF1 and EF2. Therefore, Proposition~\ref{GorBaLemma} implies that $(\varphi_{s,t})$
is a Goryainov--Ba evolution family of order~$d=\infty$. Then, according to~\cite[Lemma~3.2]{SMP}, $(f_t)$ is
one of the Loewner chains associated with this evolution family.

Let us take any $\nu>0$ and $\mu\ge\nu$. Denote $s:=t(\nu)$, $t:=t(\mu)$. Then the function $\varphi_{s,t}$
maps~$\UD$ onto $\UD\setminus H^{-1}\big(F_\mu^{-1}\big(C([\nu,\mu])\big)\big)$. Again by the Continuity theorem
(see e. g. \cite[p. 18]{Pommerenke-II}), we obtain that $\varphi_{s,t}$ belongs to the disk algebra~$\mathcal
A$.

However, for any $t\ge0$ and $\tau$ such that $t=t(\tau)$, the prime end $f^{p.e.}_{t}(1)=P_\tau$ contains no
reachable points. It follows that the angular limit of $f_t$ at the Denjoy--Wolff point $z=1$
of~$(\varphi_{s,t})$ does not exist for any~$t\ge0$. In particular, $f_t(\UD)$ is not locally connected.

Actually, the above statements hold for any Loewner chain~$(g_t)$
associated with the evolution family~$(\varphi_{s,t})$. Assume on
the contrary, that there exists a Loewner chain~$(g_t)$ associated
with $(\varphi_{s,t})$ and $t\ge0$ such that $g^{p.e.}_t(1)$
contains a reachable point, which we denote by~$A$. Since
by~\cite[Theorem 1.7]{SMP} $g_t=h\circ f_t$ for some univalent
holomorphic function~$h:\UD\to\Complex$, it would imply that there
is a slit~$\Gamma$ in~$g_t(\UD)$ landing at the point~$A\in
\partial g_t(\UD)=\partial h(f_t(\UD))\subset\overline{h(\UD)}$
such that the curve~$\Gamma_0:=h^{-1}(\Gamma)$ tends to the prime
end $f^{p.e.}_t(1)$. By~Theorem~\ref{Koebe}, the curve~$\Gamma_0$
lands at some particular point on the boundary of~$f_t(\UD)$. This
contradicts the fact that~$f^{p.e.}_t(1)$ contains no reachable
points.

Let us now summarize facts proved above in the following
\begin{proposition}\label{ex_prop}
There exists a Goryainov--Ba evolution family $(\varphi_{s,t})$, $s\in[0,+\infty)$, $t\in[s,+\infty)$, of
order~$d=\infty$ such that
\begin{itemize}
\item[(i)]~$\ell(\varphi_{s,t})=t-s$ for all $s\ge0$ and $t\ge s$;
\item[(ii)]~$\varphi_{s,t}\in\mathcal A$ for all $s\ge0$ and $t\ge
s$; \item[(iii)]~for each Loewner chain~$(f_t)$ associated to the
evolution family $(\varphi_{s,t})$ and for any ~$t\ge0$ the prime
end~$f^{p.e.}_t(1)$ is non-trivial and does not contain reachable
points; in particular $\partial  _{\infty }f_t(\UD)$ fails to be
locally connected.
\end{itemize}
\end{proposition}

\begin{remark}
Using Theorem~\ref{ChordalRepresentation_in_Examples} one can write down the Herglotz vector field~$G(z,t)$
corresponding to the evolution family~$(\varphi_{s,t})$ constructed in the above example. Namely,
\begin{equation}
G(z,t)=(1-z)^2p(z,t),\quad p(z,t):=\frac{(1-u(t))(1-z)}{4(z-u(t))},~u(t):=H(\lambda(t))\neq 1.
\end{equation}
\end{remark}

\section{Evolution families in the disk algebra}\label{EFinDA}

Proposition \ref{ex_prop} shows that the local connectivity of
$\partial_{\infty}f_t(\UD)$ for functions from a Loewner
chain~$(f_t)$ is not implied by the local connectivity of
$\partial\varphi_{s,t}(\UD)$ for functions from the evolution
family~$(\varphi_{s,t})$ generated by~$(f_t)$. The following
proposition shows that the converse implication is true.

\begin{theorem}
Let $(f_{t})$ be a Loewner chain with associated evolution family
$(\varphi _{s,t})$. Suppose that $\partial_{\infty }f_{t}(\D)$ is
locally connected for all $t\geq 0.$ Then $\varphi_{s,t}$ belongs
to the disk algebra~$\mathcal A$ for all $s\ge0$ and $t\ge s$.
\end{theorem}

\begin{proof} The argument of this proof follows the same ideas that the first two authors developed in
\cite{Contreras-Diaz:iberoamericana}.

Fix $s<t$ and suppose that there is a point $a\in \partial \D$
where the function $\varphi _{s,t}$ has no continuous extension.
Write $ U_{n}=\{z\in \D:|z-a|<1/n\}$, $n\in\Natural.$ Then there
exists a positive number $ \varepsilon $
such that $diam(\varphi _{s,t}(U_{n}))>\varepsilon $ for all $n\in\Natural$,
where $diam(\Omega )$  denotes the Euclidean diameter of a given set $%
\Omega $. Since $\varphi _{s,t}(U_{n})$ is connected, we can find
Jordan arcs $C_{n}$ in $U_{n}$ such that $diam(\varphi
_{s,t}(C_{n}))>\varepsilon $ for all $n\in\Natural.$

Moreover, since $\partial _{\infty }f_{s}(\mathbb{D})$ is locally connected, we have that
\begin{equation*}
\lim_{n\rightarrow +\infty }diam^{\sharp }(f_{s}(C_{n}))=0,
\end{equation*}
where $diam^{\sharp }(\Omega )$ denotes the spherical diameter of a given set $%
\Omega $.
Since $f_{s}(U_{n+1})\subset f_{s}(U_{n})$ for all $n\in\Natural$, there is a point $%
c\in f_{s}(\mathbb{D}) \cup \partial _{\infty }f_{s}(\mathbb{D}) $
such that
\begin{equation*}
\sup_{w\in f_{t}(\varphi _{s,t}(C_{n}))}\chi(w,c)=\sup_{w\in
f_{s}(C_{n})}\chi(w,c)\leq \sup_{w\in
f_{s}(U_{n})}\chi(w,c)\underset{n\rightarrow \infty
}{\longrightarrow }0,
\end{equation*}%
where $\chi(\cdot,\cdot)$ stands for the spherical distance.

Therefore, the arcs $\varphi _{s,t}(C_{n})$ have diameter bigger than $%
\varepsilon $ for all $n\in\Natural$ and $$\lim_{n}\sup_{w\in f_{t}(\varphi _{s,t}(C_{n}))}\chi(w,c)=0.$$
That is, $(\varphi _{s,t}(C_{n}))_{n}$ is a sequence of Koebe arcs for the function $f_{t}.$ This is a
contradiction because a univalent function is normal (see, e.\,g., \cite[Lemma~9.3 on
p.\,262]{Pommerenke}), but by a theorem of Bagemihl and Seidel (see, e.\,g., \cite[Corollary~9.1 on
p.\,267]{Pommerenke}) a non-constant normal function has no sequences of Koebe arcs. This completes the
proof.
\end{proof}

\section*{Acknowledgement}
The initial draft of this paper was prepared during a visit of the third author to the
University of Seville, which was financially supported by the project MTM2006-14449-C02-01.
The third author would also like to thank the National Center for Theoretical Sciences,
Hsinchu, Taiwan, where a substantial part of the work under the paper was carried out.

\end{document}